\documentclass[11pt]{article}

\usepackage[T1]{fontenc}
\usepackage{newtxtext, newtxmath}

\usepackage[bottom]{footmisc} 
\usepackage{amsthm, amsmath, amsfonts, amscd} 
\usepackage{thmtools, thm-restate}
\usepackage{tikz,tikz-cd,pgf, graphicx, pgfplots, float}
\usepackage[hypcap=false]{caption}
\usepackage{tabu, multirow, tabulary}
\usepackage{enumerate} 
\usepackage[shortlabels]{enumitem}
\usepackage{indentfirst} 
\usepackage{xspace} 
\usepackage[cal=boondoxo]{mathalfa}
\usepackage{hyperref} 

\hypersetup{pdfpagemode=UseNone} 

\declaretheorem[name=Theorem,parent=section]{theorem}
\declaretheorem[numbered=no,name=Theorem]{theorem*}

\declaretheorem[name=Corollary,sibling=theorem]{corollary}
\declaretheorem[numbered=no,name=Corollary]{corollary*}

\declaretheorem[name=Claim]{claim}

\declaretheorem[name=Lemma,sibling=theorem]{lemma}
\declaretheorem[numbered=no,name=Lemma]{lemma*}

\declaretheorem[name=Proposition,sibling=theorem]{proposition}
\declaretheorem[numbered=no,name=Proposition]{proposition*}

\declaretheorem[name=Definition, style=definition, sibling=theorem]{definition}

\newcommand{\Mod}[1]{\ (\mathrm{mod}\ #1)}

\def\N{\ensuremath{\mathbb N}\xspace} 
\def\R{\ensuremath{\mathbb R}\xspace} 
\def\F{\ensuremath{\mathcal F}\xspace} 
\def\T{\ensuremath{\mathcal t}\xspace} 

\def\Frect{\ensuremath{\mathcal{F}_{\sqcap}}\xspace} 
\def\Fhang{\ensuremath{\mathcal{F}_{\mathrm{hanging}}}\xspace} 
\def\Fint{\ensuremath{\mathcal{F}_{\mathrm{int}}}\xspace} 
\def\Funit{\ensuremath{\mathcal{F}_{\mathrm{unit}}}\xspace} 

\title{%
  Colouring bottomless rectangles and arborescences%
}

\author{%
  Jean Cardinal,
  Kolja Knauer,
  Piotr Micek,\\
  D\"om\"ot\"or P\'alv\"olgyi,
  Torsten Ueckerdt,
  Narmada Varadarajan
}

\tikzstyle{vertex} = [fill,shape=circle,node distance=40pt, inner sep=0pt,minimum size=3pt]
\tikzstyle{edge} = [fill,opacity=.2,fill opacity=.2,line cap=round, line join=round, line width=10pt]
\tikzstyle{elabel} =  [fill,shape=circle,node distance=30pt]
\pgfdeclarelayer{background}
\pgfsetlayers{background,main}
\pgfplotsset{compat=1.6}

\begin{document}

\maketitle

\begin{abstract}

We study problems related to colouring bottomless rectangles.
One of our main results shows that for any positive integers $m, k$, there is no semi-online algorithm that can $k$-colour bottomless rectangles with disjoint boundaries in increasing order of their top sides, so that any $m$-fold covered point is covered by at least two colours.
This is, surprisingly, a corollary of a stronger result for arborescence colourings.
Any semi-online colouring algorithm that colours an arborescence in leaf-to-root order with a bounded number of colours produces arbitrarily long monochromatic paths.
This is complemented by optimal upper bounds given by simple online colouring algorithms from other directions.

Our other main results study configurations of bottomless rectangles in an attempt to improve the \textit{polychromatic $k$-colouring number}, $m_k^*$.
We show that for many families of bottomless rectangles, such as unit-width bottomless rectangles, $m_k^*$ is linear in $k$.
We also present an improved lower bound for general families: $m_k^* \geq 2k-1$.
\end{abstract}	

\section{Introduction}
The systematic study of polychromatic colourings and cover-decomposition of geometric ranges was initiated by Pach over 30 years ago \cite{P80,P86}.
The field has gained popularity in the new millennium, with several breakthrough results; for a (slightly outdated) survey, see \cite{survey}, or see the up-to-date interactive webpage \url{http://coge.elte.hu/cogezoo.html} (maintained by Keszegh and P\'alv\"olgyi).

A family of geometric regions $\F$ and a point set $P$ in some $\R^d$ naturally define a primal hypergraph, $H(P, \F)$.
The vertex set of $H$ is the points of $P$, where $\F' \subset \F$ is an edge if for some $p \in P$, $p$ is covered by exactly the regions in $\F'$.
We are interested in the dual hypergraph, $H(\F, P)$, on the vertex set $\F$, where $p \in P$ is an edge if for some $\F' \subset \F$, $p$ is covered by exactly the regions in $\F'$.

\begin{figure}[H]
    \centering{
    \begin{tikzpicture}[scale=1.2]
    	\node[circle, fill=black,inner sep=0pt,minimum size=3pt] at (0,0) {};
	
	\node[circle, fill=black,inner sep=0pt,minimum size=3pt] at (1,1) {};

	\node[circle, fill=black,inner sep=0pt,minimum size=3pt] at (1.5,0.5) {};
	
	\draw (-0.5,0) circle (0.75cm);
	
	\draw (0.5,0.5) circle (0.8cm);
	
	\draw (1.5,0.75) circle (1cm);
    
    \node[left] at (-1.25,0) {\footnotesize $C_1$};
    
    \node[above] at (0.3,1.2) {\footnotesize $C_2$};
   
    \node[right] at (2.5,0.75) {\footnotesize $C_3$};

    \end{tikzpicture}
    \hspace{0.3cm}
      \begin{tikzpicture}
    \node[vertex,label=above:\( C_1 \)] (C1) {};
    \node[vertex, right of=C1, label=above:\( C_2 \)] (C2) {};
    \node[vertex, below of=C2, label=left:\( C_3 \)] (C3) {};
    \node[above of=C3] (C4) {};
    \node[below of=C4] (C5) {};
    \begin{pgfonlayer}{background}
    \draw[edge, color=blue, opacity = 0.2] (C1)--(C2);
    \draw[edge,color=blue, opacity = 0.2] (C3) -- (C5);
    \draw[edge, color=blue, opacity = 0.2] (C2)--(C3);
    \end{pgfonlayer}
    \end{tikzpicture}
 
    \caption{\small{A family of circles, a finite point set, and the corresponding dual hypergraph.}}
    }
    \label{fig:circles}
\end{figure}
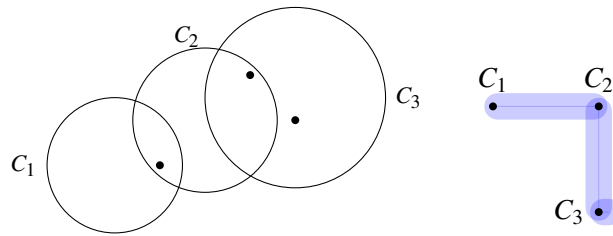

We can then define the chromatic number $\chi_{\F}$ of any family $\F$ of geometric regions. 
This is the minimum number of colours needed to colour any \textit{finite} point set $P$ so that any region containing at least two points contains at least two colours.
The dual chromatic number, $\chi_{\F}^*$, is defined analogously as the number of colours needed to properly colour any \textit{finite} subfamily of $\F$.
In this paper, we will study the \textit{polychromatic colouring numbers}.
\begin{definition}
The $k$-th (primal) \textit{polychromatic colouring number} $m_k(\F)$ is the smallest number needed to $k$-colour any finite point set $P$ so that any region containing at least $m_k(\F)$ points contains all $k$ colours.

Its dual $m_k^*(\F)$ is the smallest number needed to $k$-colour any finite subfamily of $\F$ so that any point covered by $m_k^*$ regions is covered by all $k$ colours. 
\end{definition}
The polychromatic colouring problem is partly motivated by the sensor cover problem; given a set of sensors covering an area, can we decompose them into $k$ sets so that any area covered by $m_k^*$ sensors is covered in each of these sets?
When $k=2$, this is called the cover-decomposability problem.
In particular, we say a set $P \subset \R^2$ is \textit{cover-decomposable} if $m_2^*(\F_P) < \infty$, where $\F_P$ is the family of all \textit{translates} of $P$.
In \cite{P86} it was shown that every centrally-symmetric open convex polygon is cover-decomposable, and this was extended to all open convex polygons in \cite{PT10}.
The bound $m_k^*(\F_P) = O(k)$ was proved for any convex polygon $P$ in \cite{GV09}.

The problem becomes more complicated if we consider \textit{homothets} of a convex polygon.
For example, if $\F_{\square}$  
denotes the family of axis-parallel squares in the plane, $m_2(\F_{\square} ) \leq 215$ \cite{AKV}. 
On the other hand, $m_2^*(\F_{\square} ) =\infty$, that is, for any number $m$ there is a family $\F_m$ of axis-parallel squares such that (1) each point in the plane is covered by at least $m$ squares, but (2) any $2$-colouring of $\F_m$ produces a point covered by squares of exactly one colour \cite{K13}.
Furthermore, if $\F_{\sqsubset\!\sqsupset}$ denotes the family of axis-parallel rectangles in the plane, $m_2(\F_{\sqsubset\!\sqsupset}) = \infty$ \cite{CPST}. 
Consequently, $m_k(\F_{\sqsubset\!\sqsupset}) = \infty$ for any $k$.
The dual $m_k^*(\F_{\sqsubset\!\sqsupset})$ is infinite as well; there is a constant $C>0$ such that for any numbers $n \geq r \geq 2$, there is a family of $n$ axis-parallel rectangles for which any colouring with at most $C\log n (r\log r)^{-1}$ colours produces a point covered by $r$ monochromatic axis-parallel rectangles\cite{PT08}.

This paper focuses on one particular family: \textit{bottomless rectangles}.
\begin{definition}
A subset of $\R^2$ is called a (closed) bottomless rectangle if it is of the form $\{(x,y): l \leq x \leq r, y \leq t \}$. 
We simply refer to a bottomless rectangle by these paramaters $(l,r,t)$.
\end{definition}
These range spaces were first defined by Keszegh \cite{Khalf}, who showed $m_2=4$ and $m_2^*=3$.
Later Asinowski et al.~\cite{A+13} showed that for any positive integer $k$, any finite set of points in $\R^2$ can be $k$-coloured such that any bottomless rectangle with at least $3k-2$ points contains all $k$ colors.
They also showed that the optimal number that can be written in place of $3k-2$ in the above statement is at least $1.67k$.
In our language, if $\F_{\sqcap}$ denotes the family of all bottomless rectangles in the plane, $1.67 k \leq m_k(\F_{\sqcap}) \leq 3k-2$.

Our paper studies the dual problem: we would like to determine the optimal $m_k^*(\F_{\sqcap})$.
About this question much less is known; $m_2^* = 3$ \cite{Khalf}, while the best general upper bound is $m_k^*=O(k^{5.09})$, a corollary of a more general result \cite{CKMU} about \emph{octants} (combined with an improvement of the base case \cite{octantnine} that slightly lowered the exponent).
The general conjecture, however, is that $m_k^*=O(k)$ for any family for which $m_2^*$ is finite \cite{survey}.
It was also proved in \cite{CKMU} that there is no semi-online algorithm ``from above'' for colouring bottomless rectangles.
One of our main results is a generalisation of this negative statement.

\begin{restatable}{theorem}{thmrect}\label{thm:rect}
For any $k$ and $m$, and any semi-online algorithm that $k$-colours bottomless rectangles from below (resp.\ from above, from the right, or from the left), there is a family of bottomless rectangles such that the algorithm will produce an $m$-fold covered point that is covered by at most one colour.
\end{restatable}

Our proof is much more complicated than the one in \cite{CKMU}; while there an Erd\H os-Szekeres \cite{ESz} type incremental argument is used, we need a certain diagonalisation method.
In particular, we reduce the semi-online bottomless rectangle colouring problem to a question about semi-online colourings of arborescences, which is interesting in its own right.

\begin{restatable}{theorem}{thmarbone}\label{thm:arb1}
For any $k$ and $m$, and any semi-online colouring algorithm that $k$-colours the vertices of an arborescence in a leaf-to-root order, there is an arborescence that has a leaf-to-root order such that the algorithm will produce a directed path of length $m$ that contains at most one colour.
\end{restatable}

We apply this theorem to four natural configurations of bottomless rectangles to show that for each configuration, there is a direction from which a semi-online algorithm fails.
Furthermore, we show in \autoref{thm:noshallowhittingset} that bottomless rectangles do not admit \textit{shallow hitting sets}, which are another standard tool to bound polychromatic colouring numbers.

These negative results are complemented by optimal upper bounds given by online algorithms from the other directions.
We obtain linear bounds for $m^*_k$ for the following families of bottomless rectangles:
\begin{itemize}
    \item $m^*_k(\Funit) \leq 2k-1$ for \textit{unit-width} families, \autoref{prop:stepspt},
    \item $m^*_k(\Fhang)\leq 2k-1$ for \textit{hanging} families, \autoref{thm:hanging}, and
    \item $m^*_k(\Fint)\leq 3k$ for \textit{intersecting} families. \autoref{thm:intersecting},
\end{itemize}

We also improve the current lower bound for $m^*_k$, by showing $m^*_k \geq 2k-1$ (\autoref{thm:lb}).


In \autoref{sec:configs}, we introduce certain configurations of bottomless rectangles, and define the corresponding colouring problem.
We also prove our other main results: that for many families of bottomless rectangles, such as \textit{unit-width}\footnote{See \autoref{sec:unitwidth}} and \textit{intersecting}\footnote{See \autoref{sec:other}} families, $m_k^*$ is linear.
In \autoref{sec:arb}, we prove \autoref{thm:arb1} and deduce \autoref{thm:rect} as a corollary of it.
In \autoref{sec:further}, we look at other methods to improve the upper bound on $m_k^*$, and improve the lower bound for general families to $m_k^* \geq 2k-1$.

\section{Configurations of bottomless rectangles}\label{sec:configs}

\subsection{\texorpdfstring{Erd{\H o}s}{Erdos}-Szekeres configurations}\label{sec:esconfigs}

We would like to improve the upper bound $m_k^*(\Frect) = O(k^{5.09})$ for general families by classifying some configurations of bottomless rectangles, finding a colouring for each configuration, and combining these to obtain a good colouring for general families.
To this end, we will use the classical result of Erd{\H o}s and Szekeres \cite{ESz} that any sequence of length $(k-1)^2+1$ contains a monotone subsequence of length $k$.

Recall that we associated to each rectangle its parameters $(l,r,t)$. We refer to $l$ as its left-coordinate, $r$ its right-coordinate, and $t$ its height.
Let $p$ be a point covered by $(m-1)^4+1$ rectangles.
Ordering these rectangles by left endpoint, we find a subsequence of length $(m-1)^2+1$ whose right endpoints are monotone.
Applying the result of  Erd{\H o}s and Szekeres again, we find a (sub)subsequence of length $m$ whose heights are monotone.
This proves that any point that is contained in $(m-1)^4 +1$ bottomless rectangles is contained in $m$ bottomless rectangles such that each of the three parameters of these $m$ bottomless rectangles are in increasing or decreasing order.
We name these configurations, respectively, \emph{increasing/decreasing steps, towers} and \emph{nested rectangles} (see \autoref{fig:configs}).

We are interested in colouring families with respect to a \textit{fixed} configuration.
For example, can we $k$-colour a finite family $\F$ so that any point covered by an $m$-tower is covered by all $k$ colours?
We refer to least such $m$ as $m_k^*$ \textit{for towers}, and analogously for the other configurations.
\begin{theorem}\label{thm:configs}
$m_k^* = k$ for each fixed configuration.
\end{theorem}
A result of Berge \cite{Berge} shows that for any family $\F$ of geometric regions, $m_2^*(\F) = 2$ if and only if $m_k^*(\F) = k$ for all $k$.
It is not hard to show that $m_2^* = 2$ for each configuration, and then apply the result of Berge.
Nevertheless, it is valuable to see that there is a simple online algorithm for each configuration.

\begin{figure}[H]
	\centering{	
		\begin{tikzpicture} [y=0.4 cm, x=0.5cm]
		\draw (0,-2)--(0,0) -- (4,0)--(4,-2);
		
		\draw (1,-2)--(1,1) -- (5,1)--(5,-2);
		
		\draw[dotted] (1.5,-2)--(1.5,2) -- (5.5, 2)--(5.5,-2);
		
		\draw (2,-2)--(2,3) -- (6,3)--(6,-2);
		
		\node[left] at (0,0) {\footnotesize $R_1$};
		\node[left] at (1,1) {\footnotesize $R_2$};
		\node[left] at (2,3) {\footnotesize $R_m$};
		
		\node[align=center, below of=1cm,]at (3,-1) {\footnotesize \textit{increasing m-steps}} ;
		
		\end{tikzpicture}
		\hspace{0.3cm}
		\begin{tikzpicture} [y=0.4 cm, x=0.5cm]
		\draw (0,-2)--(0,3) -- (4,3)--(4,-2);
		
		\draw (1,-2)--(1,2) -- (5,2)--(5,-2);
		
		\draw[dotted] (1.5,-2)--(1.5,1) -- (5.5, 1)--(5.5,-2);
		
		\draw (2,-2)--(2,0) -- (6,0)--(6,-2);
		
		\node[left] at (0,3) {\footnotesize $R_1$};
		\node[left] at (1,2) {\footnotesize $R_2$};
		\node[left] at (2,0) {\footnotesize $R_m$};
		
		\node[align=center, below of=1cm,]at (3,-1) {\footnotesize \textit{decreasing m-steps}} ;
		
		\end{tikzpicture}
		
		\hspace{0.3cm}
		\begin{tikzpicture} [y=0.4 cm, x=0.5cm]
		\draw (0,-2)--(0,0) -- (6,0)--(6,-2);
		
		\draw (1,-2)--(1,1) -- (5,1)--(5,-2);
		
		\draw[dotted] (1.5,-2)--(1.5,2) -- (4.5, 2)--(4.5,-2);

		\draw (2,-2)--(2,3) -- (4,3)--(4,-2);
		
		\node[left] at (0,0) {\footnotesize $R_1$};
		\node[left] at (1,1) {\footnotesize $R_2$};
		\node[left] at (2,3) {\footnotesize $R_m$};
		
		\node[align=center, below of=1cm,] at (3,-1) {\footnotesize \textit{m-tower}} ;

		\end{tikzpicture}	
		\hspace{0.3cm}
		\begin{tikzpicture} [y=0.4 cm, x=0.5cm]
		\draw (0,-2)--(0,3) -- (6,3)--(6,-2);
		
		\draw (1,-2)--(1,2) -- (5,2)--(5,-2);
		
		\draw[dotted] (1.5,-2)--(1.5,1) -- (4.5, 1)--(4.5,-2);

		\draw (2,-2)--(2,0) -- (4,0)--(4,-2);
		
		\node[left] at (0,3) {\footnotesize $R_1$};
		\node[left] at (1,2) {\footnotesize $R_2$};
		\node[left] at (2,0) {\footnotesize $R_m$};
		
		\node[align=center, below of=1cm,] at (3,-1) {\footnotesize \textit{m-nested rectangles}} ;

		\end{tikzpicture}
		\caption{\small{Erd{\H o}s-Szekeres configurations}}\label{fig:configs}	
	}	
\end{figure}
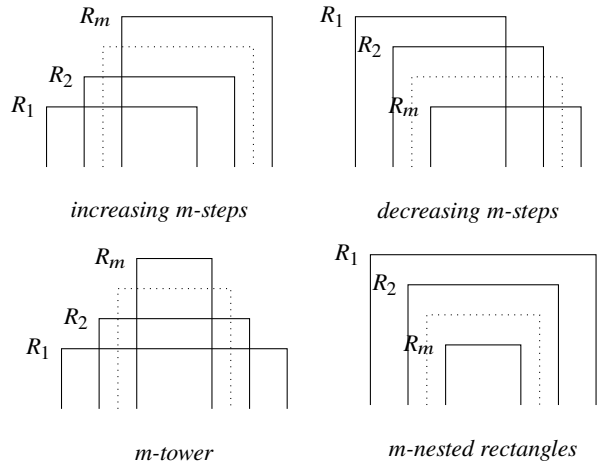

\begin{definition}
In this paper, we will consider the following two types of colouring algorithms for hypergraphs, which receive the vertices in some order.
Both types must colour the vertices \textit{irrevocably} -- they are not allowed to recolour vertices.
\begin{enumerate} [(1)]
    \item An \emph{online algorithm} must colour each vertex immediately so that at each step, there is no conflict in the partial colouring. \footnote{There is a large literature on online algorithms \cite{GKMZ,HSz,LST}.}
    \item A \textit{semi-online algorithm} need not colour each vertex immediately, but must ensure that at each step there is no conflict in the partial colouring.
\end{enumerate}
\end{definition}

This condition on the partial colouring means that, for example, at every step the algorithm must maintain that a point covered by $m$ rectangles is covered by all $k$ colours, but not all points have to be colored.

\begin{proof}[Proof of \autoref{thm:configs}]
We first present a colouring algorithm for towers. We colour the rectangles in increasing order of height, i.e.\ from below, so that at every step the following property holds.

\begin{enumerate}
\item[(*)]	If a point $p$ is covered by a $j$-tower for $j \leq k$, then $p$ is covered by at least $j$ different colours.
\end{enumerate}

At step $1$, colour the rectangle of least height arbitrarily.
Suppose the first $t-1$ rectangles have been coloured so that (*) holds. We colour the rectangle $R_t$ as follows.
For each $1 \leq i \leq k$, let $y_i$ be the largest number so that if $p \in R_t$ has $y$-coordinate less than $y_i$, then $p$ is covered by colour $i$. (This corresponds to a tallest rectangle $S$ of colour $i$ such that ($S,R_t$) is a tower.) If $y_i$ does not exist for some colour $i$, colour $R_t$ with colour $i$. Otherwise, suppose $y_1 > \dots > y_k$, and colour $R_t$ with colour $k$. 

Property (*) holds: if $p$ is covered by a $j$-tower, then either $p$ was already covered by $j$ colours, or we added a new colour to the set of colours covering $p$.
We use the same algorithm to colour $k$-nested sets, only we colour the rectangles from above. 
It is easy to check that with this ordering, the same property holds.

\vspace{0.3cm}
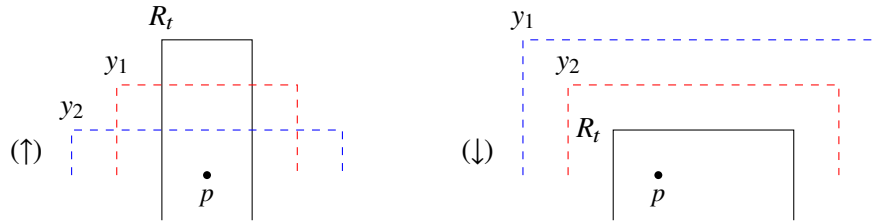
\begin{figure}[H]
	\centering {
		\begin{tikzpicture}[x=0.6cm,y=0.6cm]
		\draw (0,0)--(0,4)--(2,4)--(2,0);
		\draw[dashed, color=red] (-1,1)--(-1,3)--(3,3)--(3,1);
		\draw[dashed, color=blue] (-2,1)--(-2,2)--(4,2)--(4,1);
		
		\node[above] at (-3,1) {($\uparrow$)};
		\node[above] at (0,4) {$R_t$};
		\node[above] at (-1,3) {$y_1$};
		\node[above] at (-2,2) {$y_2$};
		
		\node[circle, fill=black,inner sep=0pt,minimum size=3pt,label=below:{\small$p$}] at (1,1) {};

		\draw (10,0)--(10,2)--(14,2)--(14,0);
		\draw[dashed, color=red] (9,1)--(9,3)--(15,3)--(15,1);
		\draw[dashed, color=blue] (8,1)--(8,4)--(16,4)--(16,1);
		
		\node[above] at (7,1) {($\downarrow$)};
		\node[left] at (10,2) {$R_t$};
		\node[above] at (9,3) {$y_2$};
		\node[above] at (8,4) {$y_1$};
		
		\node[circle, fill=black,inner sep=0pt,minimum size=3pt,label=below:{\small$p$}] at (11,1) {};

		\end{tikzpicture}
		\caption{\small{Colouring algorithms for $k$-towers and $k$-nested sets respectively}}	
	}
\end{figure}
\vspace{0.3cm}
The algorithm for increasing $k$-steps is only slightly different. We colour the rectangles in decreasing order of right endpoint (from the right). At step $t$, for $1 \leq i \leq k$, let $x_i$ be the least number so that if $p \in R_t$ has $x$-coordinate greater than $x_i$, then $p$ is covered by a rectangle of colour $i$ (corresponding to the leftmost rectangle $S$ of colour $i$ such that ($R_t,S$) form increasing steps). As earlier, if some $x_i$ does not exist, give $R_t$ colour $i$. Otherwise, if $x_1 < \dots < x_k$, give $R_t$ colour $k$.
\vspace{0.3cm}
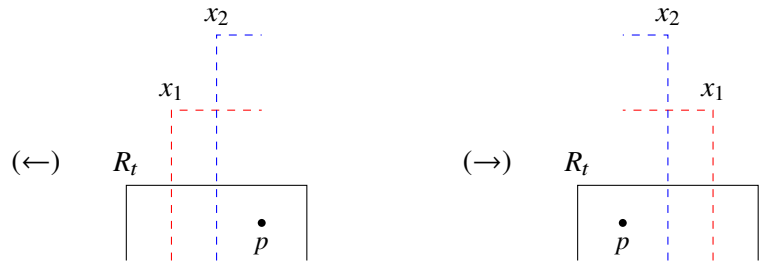
\begin{figure}[H]
	\centering{
		\begin{tikzpicture}[x=0.6cm]
		\draw (0,0)--(0,1)--(4,1)--(4,0);
		\draw[dashed, color=red] (1,0)--(1,2)--(3,2);
		\draw[dashed, color=blue](2,0)--(2,3)--(3,3);
		
		\node[above] at (0,1) {$R_t$};
		\node[above] at (1,2) {$x_1$};
		\node[above] at (2,3) {$x_2$};
		\node[above] at (-2,1) {($\leftarrow$)};
		
		\node[circle, fill=black,inner sep=0pt,minimum size=3pt,label=below:{\small$p$}] at (3,0.5) {};

		\draw (10,0)--(10,1)--(14,1)--(14,0);
		\draw[dashed, color=red](13,0)--(13,2)--(11,2);
		\draw[dashed, color=blue](12,0)--(12,3)--(11,3);
		\node[above] at (10,1) {$R_t$};
		\node[above] at (13,2) {$x_1$};
		\node[above] at (12,3) {$x_2$};
		\node[above] at (8,1) {($\rightarrow$)};
		
		\node[circle, fill=black,inner sep=0pt,minimum size=3pt,label=below:{\small$p$}] at (11,0.5) {};

		\end{tikzpicture}
		\caption{\small{Colouring algorithms for increasing and decreasing $k$-steps respectively}}
	}	
\end{figure}
\end{proof}
Note that ordering a tower in increasing order of height (from below) is the same as ordering it in decreasing order of right endpoint (from the right), or increasing order of left endpoint (from the left).
We may repeat the same algorithm for towers, colouring the rectangles from the left (or right) and we will still obtain a good colouring.
Similarly, ordering a nested set from above is the same as ordering it from the left or right.
Ordering increasing steps from the right (resp.\ decreasing steps from the left) is the same as ordering them from above.
This is all to say that the same algorithm can be used from these ``good'' directions for each fixed configuration.

\begin{center}
	\bgroup
	\def\arraystretch{1.5}
	\begin{tabular}[t]{|c|c|c|c|c|}
		\hline 
		
		& \small{left ($\rightarrow$)} & \small{right ($\leftarrow$)} & \small{below ($\uparrow$)} & \small{above ($\downarrow$)}\\
		\hline
		\small{inc.\ steps} & \small{$\infty$} & \small{$=k$} & \small{$\infty$} & \small{$=k$} \\
		\hline
		
		\small{dec.\ steps} & \small{$=k$} &\small{$\infty$} &\small{$\infty$}&\small{$=k$}\\
		\hline
		
		\small{towers} & \small{$=k$} & \small{$=k$} & \small{$=k$} &  \small{$\infty$}  \\
		\hline
		
		\small{nested} & \small{$=k$} & \small{$=k$} & \small{$\infty$} & \small{$=k$} \\
		\hline
\end{tabular}\label{tab:configs}
\captionof{table}{\small{$m_k^*$ values for each configuration given by semi-online algorithms from different directions.}}
\egroup
\end{center}

The value $\infty$ indicates the non-existence of semi-online colouring algorithms, which we prove in the next section.

The next natural question to ask is: how can we combine these colourings?
Nested rectangles seem to have the ``simplest'' structure of the four configurations.
Indeed, ordering nested rectangles from above is the same as ordering them from the left or right. 
Further, if ($R_1, R_2$) are nested ($R_1$ contains $R_2$), then any point in $R_2$ is neccessarily contained in $R_1$. 
This shows that we can modify the algorithms for the other configurations to also colour nested rectangles.

\begin{proposition}
	$m_k^*=k$ if we $k$-colour any finite family $\mathcal{F}$ with respect to 
	\begin{enumerate}[(a)]
		\item towers and nested sets
		\item increasing steps and nested sets
		\item decreasing steps and nested sets
	\end{enumerate}
\end{proposition} 
\begin{proof} We first present the algorithm for towers and nested sets. The precise statement is that any finite family $\mathcal{F}$ can be $k$-coloured so that any point contained in a $k$-tower or a $k$-nested set is covered by all $k$ colours.
We colour the rectangles from the right (this can also be done from the left). We maintain the same property as earlier.
	
	\begin{enumerate}
	\item[(*)]	If a point $p$ is covered by a $j$-tower or a $j$-nested set for $j \leq k$, then $p$ is covered by at least $j$ different colours.
	\end{enumerate}
	
	At step $t$, for $1 \leq i \leq k$, let $y_i$ be the greatest number so that if $p \in R_t$ has $y$-coordinate less than $y_i$, then $p$ is covered by a rectangle of colour $i$. 
	As earlier, if some $y_i$ does not exist, give $R_t$ colour $i$. Otherwise, suppose $y_1 > \ldots > y_k$, and give $R_t$ colour $k$.
	To prove that (*) holds is not as straightforward as in \autoref{thm:configs}. 
	Let $y$ denote the height of $R_t$. 
	If $y_k > y$, or $y > y_1$, (*) holds by the same argument as in \autoref{thm:configs}. 
	If not, we have $y_1 > \dots > y_{l-1} > y > y_l > \dots > y_k$. 
	Suppose $p \in R_t$ is covered by a $j$-nested set $R_1, \dots, R_{j-1}, R_t$. 
	Since each $y_i$ is maximal, (*) holds by the same argument as earlier. 
	The only essentially different case is when $p \in R_t$ is covered by a $j$-tower. 
	If we did not add a new colour to the set of rectangles containing $p$, this means that $p$ was already covered by a rectangle of colour $k$. However, as $y_k$ was chosen to be maximal, the $y$-coordinate of $p$ must be less than $y_k$, so $p$ is already covered by all $k$ colours. 
	
	The algorithms for increasing and decreasing steps are modified in the exact same way. \end{proof}

\subsection{Unit-width rectangles}\label{sec:unitwidth}
The next natural pair of Erd{\H o}s-Szekeres configurations to attempt to combine is the increasing and decreasing steps.
For this, we consider a different setup; refer to \textit{steps} as the case when we assume that $\F$ does not contain any towers or nested sets.
What is $m_{k, \text{steps}}^*$?

Let \textit{unit bottomless} be the case when all the rectangles in $\F$ have the same width, or unit width.  
It is clear that ``\textit{unit bottomless} $\subset$ \textit{steps}'', as any family of unit bottomless rectangles cannot contain towers or nested sets.

\begin{proposition}
``\textit{steps} $=$ \textit{unit bottomless}'', i.e.\ any family of tower- and nested set-free rectangles can be realised as a family of unit width bottomless rectangles so that the corresponding dual hypergraphs are isomorphic.
\end{proposition}
\begin{proof}
We prove the inclusion steps $\subseteq$ unit bottomless by our favourite method, induction on $|\mathcal{F}|$. Suppose any family $\mathcal{F}$ of $n-1$ rectangles that do not contain towers or nested sets can be realised as a family $\mathcal{F}_{unit}$ of unit bottomless rectangles (with an isomorphic hypergraph), and that this realisation preserves heights and the ordering of left endpoints. That is, the height of a rectangle $R$ in $\mathcal{F}$ is the same as its realisation in $\mathcal{F}_{unit}$.

 Let $|\mathcal{F}| = n$, and let $R$ be the leftmost rectangle in $\mathcal{F}$. Take any realisation of $\mathcal{F}\setminus R$ as a family $\mathcal{G}$, and let $R_1, \dots, R_m$ be the rectangles that intersect $R$, and $R_1', \dots, R_m'$ their realisations. Assume without loss of generality that $l(R_1) < \dots < l(R_m)$.
 
In particular, $l(R_m) < r(R) < r(R_1)$ (as they intersect), so $R_1, \dots, R_m$ also intersect each other. This implies that the interval $[l(R_1'), \dots, l(R_m')]$ has length strictly less than $1$. Thus for $\epsilon$ small enough, if we realise $R$ as a unit width rectangle $R'$ with $r(R') = l(R_m') + \epsilon$ with the same height, then $R'$ will intersect exactly the rectangles $R_1', \dots, R_m'$ (with the same hypergraph structure).
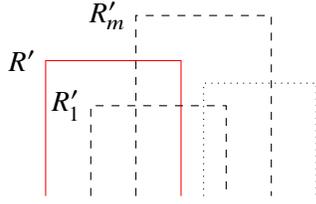
\begin{figure}[H]
\begin{center}
\begin{tikzpicture}[x=0.6 cm, y = 0.6cm]
\draw[color=red] (0,0)--(0,3)--(3,3)--(3,0);
\draw[dashed] (1,0)--(1,2)--(4,2)--(4,0);
\draw[dashed](2,0)--(2,4)--(5,4)--(5,0);
\draw[dotted](3.5,0)--(3.5,2.5)--(6,2.5)--(6,0);

\node[left] at (0,3) {$R'$};
\node[left] at (1,2) {$R_1'$};
\node[left] at (2,4) {$R_m'$};

\end{tikzpicture}
\end{center}
\caption{\small{We ensure that the realisation of $R$ preserves the hypergraph structure.}}
\end{figure}
\end{proof}

So instead of considering colouring points with respect to bottomless rectangles, we may consider colouring steps with respect to points.
\begin{proposition} \label{prop:stepspt}
	For steps, $m_k^* \leq 2k-1$. 
\end{proposition}
The proof of the proposition will use ABA-free hypergraphs \cite{ABA}. We say a hypergraph $\mathcal{H}$ with an ordering $<$ of its vertex set is ABA-free if there are no hyperedges $A$ and $B$ and vertices $x < y < z$ with $x,z \in A \setminus B$ and $y \in B \setminus A$. For example, interval hypergraphs - where the vertices are points in $\mathbb{R}$ and the hyperedges are the subsets induced by some intervals - are ABA-free.
A result of \cite{ABA} tells us that for ABA-free hypergraphs, $m_k \leq 2k-1$.

\begin{proof}[Proof of \autoref{prop:stepspt}]
Let $\mathcal{F}$ be a family containing no nested sets or towers and $P$ a finite point set. We claim that by ordering the rectangles by left endpoint, the resulting hypergraph on the vertex set $\mathcal{F}$ with edges induced by $P$ is ABA-free. Suppose for contradiction we have three rectangles with $l(R_1) < l(R_2) < l(R_3)$, and points $p$ and $q$ so that $p \in (R_1 \cap R_3) \setminus R_2$, and $q \in R_2 \setminus (R_1 \cap R_2)$.

Recall that a point $(x,y)$ is in a rectangle $R$ if and only if $x \in [l(R), r(R)]$ and $y < y(R)$. Let $p = (x_p, y_p)$ and $q = (x_q,y_q)$.
Then, $p \in R_1, R_3$ but $p \notin R_2$ implies,
\[
l(R_1) < l(R_2) < l(R_3) < x_p < r(R_1) < r(R_2) < r(R_3)\text{, and}
\]
\[
y(R_1), y(R_3) > y_p > y(R_2)
\]
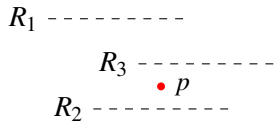
\begin{figure}[H]
\begin{center}
    \begin{tikzpicture}[x = 0.6 cm, y = 0.6 cm]
\draw[dashed] (0,3)--(3,3); 
\node[left] at (0,3) {$R_1$};
\draw[dashed] (1,1)--(4,1); 
\node[left] at (1,1) {$R_2$};
\draw[dashed] (2,2)--(5,2); 
\node[left] at (2,2) {$R_3$};
    \node[circle, fill=red,inner sep=0pt,minimum size=3pt,label=right:{\small$p$}] at (2.5,1.5) {};
    \end{tikzpicture}
\end{center}
\caption{\small{An arrangement of $p$ and the three rectangles implied by the above conditions.}}\label{fig:stepsaba}
\end{figure}

As a result, $[l(R_2), r(R_2)]$ is covered by the intervals $[l(R_1), r(R_1)]$ and $[l(R_3), r(R_3)]$, and $R_2$ is below $R_1$ and $R_3$; the ``top side'' of $R_2$ is covered by the top sides of $R_1$ and $R_3$ as in \autoref{fig:stepsaba}.
Thus any point in $R_2$ is contained in at least one of $R_1$ and $R_3$, contradicting that $q \in R_2$ but $q \notin R_1$ or $R_3$.

\end{proof}

We end this subsection by extending this to families that do not contain towers.

\begin{theorem}\label{thm:towerfree}
For families $\mathcal{F}$ that do not contain towers, $m_k^* \leq 2k-1$.
\end{theorem}
\begin{proof}
As earlier, we want to show that the corresponding hypergraph is ABA-free. Suppose again that we have three rectangles with $l(R_1) < l(R_2) < l(R_3)$, and points $p$ and $q$ so that $p \in R_1, R_3$, $q \notin R_1, R_3$, and $q \in R_2$, $p \notin R_2$.

The previous proposition shows $R_1, R_2, R_3$ must contain at least one nested set. It is also easy to see that not all three of them can form a nested set, so exactly two of them do. Further, the condition $p \in R_1, R_3$ but $q \notin R_1, R_3$ implies that $(R_1,R_3)$ must form a nested set (where $R_1$ contains $R_3$). In this case, it is easy to check that it is not possible to have a rectangle $R_2$ that forms steps with both $R_1$ and $R_3$, and contains $q$ but not $p$.
\end{proof}

\subsection{Other families}\label{sec:other}
For families that are tower-free, $m_k^*$ is linear.
What of families that do contain towers?

We say a family of bottomless rectangles is \textit{hanging} if their left endpoints lie on the line $y=x$.
It is clear that we can choose any line with positive slope, as rotating the line $y=x$ and moving the left endpoints along with it preserves the hypergraph structure.
So more generally, a \textit{hanging family} is one whose left endpoints all lie on a fixed line with positive slope, which will be $y=x$ for convenience.
A tower, for example, can be realised as a hanging arrangement by ``stretching'' the left endpoints.

\begin{proposition}\label{prop:hanging}
For \textit{hanging families}, $m_k^* \leq 3k-2$.
\end{proposition}
This proof relies on a reduction to the primal problem for general families of bottomless rectangles, when we colour points with respect to bottomless rectangles.
We want to show that any hanging family $\F$ and point set $P$ can be realised as a family $\F(P)$ and point set $P(\F)$ so that a rectangle $R \in \F$ covers a point $q \in P$ if and only if the corresponding point $r \in P(\F)$ is contained in the rectangle $Q \in \F(P)$.

To each rectangle $R \in \F$, we associate its right endpoint $r(R)$, and to each point $q=(x,y) \in P$, we associate an infinite hanging rectangle from the point $(x,x)$.
Of course, when we say infinite, we simply mean that the right endpoint of the corresponding rectangle in $\F(P)$ is sufficiently large. 

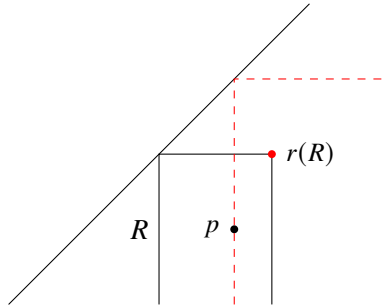
\begin{figure}[H]
    \centering{
    \begin{tikzpicture}
    \draw (-2,-2)--(2,2);
    \draw (0,-2)--(0,0)--(1.5,0)--(1.5,-2);
    \node[left] at (0,-1) {$R$};
    \node[circle, fill=black,inner sep=0pt,minimum size=3pt,label=left:{\small$p$}] at (1,-1) {};
    \draw[dashed, color=red] (1,-2)--(1,1)--(3,1);
    \node[circle, fill=red,inner sep=0pt,minimum size=3pt,label=right:{\small$r(R)$}] at (1.5,0) {};

    \end{tikzpicture}
    }
    \caption{\small{The infinite hanging rectangle from $p$ contains $r(R)$ if and only if $R$ contains $p$.}}
    \label{fig:hangingdual}
\end{figure}

Since the best upper bound for the primal problem is $3k-2$, we have the desired bound for hanging families.
However, by modifying the proof of this upper bound from \cite{A+13}, we can improve it to $2k-1$.
By the duality we observed, it suffices to consider colouring points with respect to infinite hanging rectangles.
Order the points in increasing order of $y$-coordinate (from below).
Note that if $p$ is to the left of some hanging rectangle $R$, then $R$ does not contain $p$, so once we start colouring points of $R$ we may ``disregard'' $p$.
We present the points in increasing order of $y$-coordinate as follows.
At step $t$, suppose $q$ is to be presented, and $p$ is the leftmost point at this step. 
If $q$ is covered by a rectangle to the \textit{right} of $p$, first we ``delete'' $p$, then we present $q$.
If not, then we present $q$ without deleting any points.
This enables us to only care about the leftmost $2k-1$ points, as we discard the nonessential ones.

More precisely, we frame this as a \textit{dynamic colouring problem} on the line.
We wish to colour a dynamically appearing point set $P$ where one of the following kinds of events can occur.
\begin{enumerate}[(1)]
    \item A new point appears.
    \item The leftmost point disappears.
\end{enumerate}

\begin{theorem}\label{thm:hanging}
For hanging arrangements, $m_k^* \leq 2k-1$.
\end{theorem}
\begin{proof}
Given a dynamically appearing point set $P$ as discussed, we want to $k$-colour it so that at every step, the leftmost $2k-1$ points contain at least one point of each colour.
For $i=1, \dots, k$, define an \textit{$i$-gap} to be an inclusion-maximal set of points between two points of colour $i$, and an \textit{$i$-prefix} to be the set of points before the first point of colour $i$.
It suffices to maintain the following invariants at each step.
\begin{enumerate}[(a)]
    \item All $i$-gaps have size at least $k-1$, and
    \item all $i$-prefixes have size at most $2k-2$.
\end{enumerate}
Suppose that these invariants are satisfied at some step, and then an event of type (1) occurs. 
This can only harm (b) by creating an $i$-prefix of size $2k-1$ for some $i$.
At most $k-1$ colours occur in the leftmost $k$ points, and by (a), no colour occurs more than once.
So there is some uncoloured point which we can colour with colour $i$; by construction, this is separated from the next point of colour $i$ by at least $k-1$ points, so (a) is preserved.

Now suppose an event of type (2) occurs: the first point disappears.
Again, this can only harm (b) by creating an $i$-prefix of size $2k-1$.
This means that the leftmost point had colour $i$, so we may once again find an uncoloured point in the leftmost $k$ points of the $i$-prefix, and colour it with colour $i$.
\end{proof}

Another type of family we consider is the \textit{intersecting family}. 
These are families where any two rectangles are pairwise intersecting.
In particular, there is a point $v$ contained in the intersection of all the rectangles.
We call the vertical line through $v$ the \textit{spine} of the family.

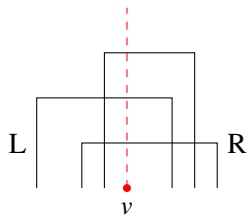
\begin{figure}[H]
    \centering{
    \begin{tikzpicture}[scale=0.6]
    \draw[dashed, color=red] (0,0)--(0,4);
    \draw (-1,0)--(-1,1)--(2,1)--(2,0);
    \draw (-2,0)--(-2,2)--(1,2)--(1,0);
    \draw (-0.5,0)--(-0.5,3)--(1.5,3)--(1.5,0);
     \node[circle, fill=red,inner sep=0pt,minimum size=3pt,label=below:{\small$v$}] at (0,0) {};
     \node[left] at (-2,1) {L};
     \node[right] at (2,1) {R};
    \end{tikzpicture}
    }
    \caption{\small{The \textit{spine} through $v$ defines the sides L and R.}}
    \label{fig:spine}
\end{figure}

\begin{theorem}\label{thm:intersecting}
For intersecting families, $m_k^* \leq 3k$.
\end{theorem}
Our proof will rely on \textit{shallow hitting sets}.
\begin{definition}
A subset $S \subset \F$ is a \textit{$c$-shallow hitting set for depth $d$} if every point that is covered by exactly $d$ rectangles is covered by at least $1$ and at most $c$ rectangles of $S$. 
\end{definition}
Suppose for fixed $c$ and a family $\F$ of bottomless rectangles, we have a $c$-shallow hitting set for any depth $d$.
First we construct a $c$-shallow hitting set $S_1$ for depth $ck$, then remove $S_1$ from $\F$ and construct a $c$-shallow hitting set $S_2$ for depth $ck-c$, then remove $S_2$ from $\F$ and so on, yielding some disjoint subfamilies $S_1, \dots, S_k$.
It is now easy to check that any point covered by $ck$ rectangles is covered by at least $1$ rectangle from each of $S_1, \dots, S_k$: $m_k^* \leq ck$ for $\F$.
\begin{proof}[Proof of \autoref{thm:intersecting}]
Let $d$ be arbitrary. 
We will construct hitting sets $S_L$ and $S_R$, and show that $S = S_L \cup S_R$ is a $3$-shallow hitting set for depth $d$.
Order the points at depth exactly $d$ from below, and we will add rectangles to $S_L$ and $S_R$ in this order.
If $p$ is a point at depth $d$ on side $s$ that is not yet covered by any rectangle of $S_s$, add to $S_s$ the rectangle covering $p$ whose extension to the other side is shortest.
Once we are finished constructing $S_L$ and $S_R$ in this order, we reduce $S = S_L \cup S_R$ so that it is a minimal hitting set for depth $d$.

To show that it is $3$-shallow, we show that every point $p$ at depth $d$ on side, say $L$, is covered by at most $2$ rectangles from $S_L$ and at most $1$ from $S_R$.
This is by minimality; if $p$ is covered by $T_1, T_2, T_3 \in S_L$, removing the one of the lower two whose left endpoint is closer to the spine preserves that $S$ is a hitting set.
(This is not true without the fact that $\F$ is intersecting.)
If $p$ is contained by two rectangles $R_1, R_2 \in S_R$, suppose that $y(R_1) > y(R_2)$.
Since we chose $R_1$ for $S_R$ from below, there must be a point $q$ on the right side that is covered by $R_1$ but is above $R_2$.
As a result, there are $d-1$ other rectangles covering $q$ whose heights are between $R_1$ and $R_2$.
By the choice of $R_2$ for $q$ by minimality of its left endpoint, the left endpoints of these $d-1$ rectangles extend beyond the left endpoint of $R_2$, so they cover $p$.
This gives $d+1$ rectangles that cover $p$, a contradiction.
\end{proof}
\begin{figure}[H]
    \begin{center}
    \begin{tikzpicture}[scale=0.7]
    \node[circle, fill=red,inner sep=0pt,minimum size=3pt,label=below:{\small$v$}] at (0,0) {};
    \draw[dashed, color=red] (0,0)--(0,5);
    \draw (-1,0)--(-1,4)--(2,4)--(2,0);
    \node[left] at (-1,4) {\footnotesize{$T_1$}};
    \draw (-2,0)--(-2,3)--(3,3)--(3,0);
    \node[left] at (-2,3) {\footnotesize{$T_2$}};
    \draw (-3,0)--(-3,2)--(1.5,2)--(1.5,0);
    \node[left] at (-3,2) {\footnotesize{$T_3$}};
    \node[circle, fill=black,inner sep=0pt,minimum size=3pt,label=below:{\small$p$}] at (-0.5,1.5) {};
    \end{tikzpicture}
    \hspace{1cm}
    \begin{tikzpicture}[scale=0.7]
    \node[circle, fill=red,inner sep=0pt,minimum size=3pt,label=below:{\small$v$}] at (0,0) {};
    \draw[dashed, color=red] (0,0)--(0,5);
    \draw (-2,0)--(-2,4)--(1,4)--(1,0);
    \node[right] at (1,4) {\footnotesize{$R_1$}};
    \draw (-1,0)--(-1,2.5)--(3,2.5)--(3,0);
    \node[right] at (3,2.5) {\footnotesize{$R_2$}};
    \node[circle, fill=black,inner sep=0pt,minimum size=3pt,label=below:{\small$p$}] at (-0.5,1.5) {};
    \node[circle, fill=black,inner sep=0pt,minimum size=3pt,label=below:{\small$q$}] at (0.5,3) {};
    \draw[dotted] (-1.5,0)--(-1.5,3.75)--(2.5,3.75)--(2.5,0);
    \node[right] at (2.5,3.7 ) {\footnotesize{$\ddots$}};
    \draw[dotted] (-2.5,0)--(-2.5,3.25)--(3.5,3.25)--(3.5,0);
    \end{tikzpicture}
    \caption{\small{(left) $T_2$ is not needed to hit points at depth $d$ on the left, and (right) there are $d-1$ rectangles between $R_1$ and $R_2$ that cover $p$.}}\label{fig:intersecting}
    \end{center}
\end{figure}
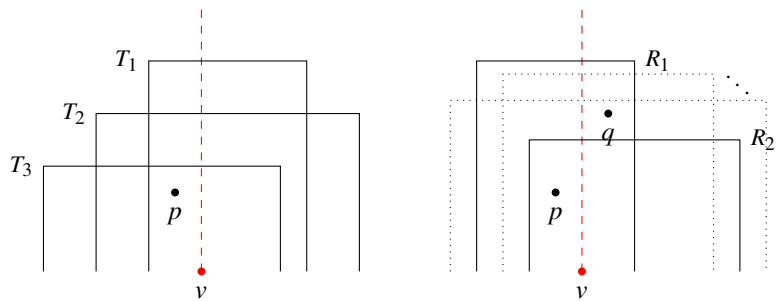
\section{Arborescences}\label{sec:arb}

The goal for this section is to prove \autoref{thm:arb1},
\thmarbone*
and to show that \autoref{thm:rect} can be deduced from it.
\thmrect*
We would like to associate to each family $\F$ of rectangles a simple graph, and derive a polychromatic colouring of $\F$ from a suitable colouring of this graph.
First we will define the family of graphs that we consider (arborescences), prove \autoref{thm:arb1}, then show how these graphs are obtained from bottomless rectangles.
\subsection{The setup}

An arborescence is a directed tree with a distinguished \textit{root} vertex such that all edges are directed away from the root, i.e.\ there is a unique directed path from the root to any vertex.
We denote the length of the shortest directed path from $u$ to $v$, if it exists, by $dist(u,v)$. 
Recall that the length of a path is the number of edges, or one less than the number of vertices.
A disjoint union of arborescences is called a \textit{branching}.
We say that an ordering of the vertices of a branching is \emph{root-to-leaf} if every vertex is preceded by its in-neighbors and succeeded by its out-neighbors; in particular, from every component first the root is presented and last a leaf.

\begin{claim}
	The vertices of any branching can be $k$-coloured by an online algorithm in a root-to-leaf order such that any directed path on $k$ vertices contains all $k$ colours.
\end{claim}
\begin{proof}
If a root is presented, colour it with colour $1$.
Every time a new vertex $v$ is presented in the component with root $r$, colour $v$ according to the parity of $dist(r,v) \Mod{k}$ (which can be determined from a root-to-leaf ordering).
\end{proof}

We call the reversal of a root-to-leaf ordering a \textit{leaf-to-root} ordering; from each component, first a leaf is presented and last the root.
Our main result, \autoref{thm:arb1}, shows that the converse of the above claim fails: any semi-online algorithm will in fact leave arbitrarily long monochromatic paths.
In order to apply this result to bottomless rectangles, however, we will need a stronger condition on the leaf-to-root ordering.

For two vertices $u$ and $v$ of a branching, say $u<v$ if they are in the same component and there is a directed path from $u$ to $v$.
This forms a partial order where the roots are the minimal elements and the leaves the maximal.
A leaf-to-root ordering is a linear extension of this partial order that presents the $<$-maximal element first.

If $u<v$ and there are no other vertices between them, i.e.\ $uv$ is a directed edge, write $u \lessdot v$ and say that $v$ is a \emph{parent} of $u$.
(Thus, somewhat contradicting the laws of nature, every vertex can have only one child, but several parents.)
When presenting the vertices of a branching in a leaf-to-root order, the newly presented vertex $u$ will always form a root, while its parents were all roots of the branching before $u$ was presented.

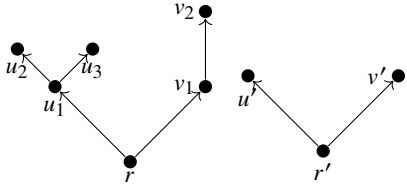
\begin{figure}[H]
	\centering{
		\begin{tikzpicture}
		[scale=1,auto=left,every node/.style={circle, minimum size = 5pt, inner sep = 0,  fill=black!}]
		\node[label=below:{\footnotesize{$r$}}] (r) at (0,0) {};
		\node[label=below:{\footnotesize{$u_1$}}] (u1) at (-1,1)  {};
		\node[label=below:{\footnotesize{$u_3$}}] (u3) at (-0.5,1.5)  {};
		\node[label=below:{\footnotesize{$u_2$}}] (u2) at (-1.5,1.5) {};
		\node[label=left:{\footnotesize{$v_1$}}] (v1) at (1,1)  {};
		\node[label=left:{\footnotesize{$v_2$}}] (v2) at (1,2)  {};	
		
		\foreach \from/\to in {r/u1, u1/u2, u1/u3, r/v1, v1/v2}
		\draw[->] (\from) -- (\to);

		\end{tikzpicture}
		\hspace{0.1cm}
		\begin{tikzpicture}
		[scale=1,auto=left,every node/.style={circle, minimum size = 5pt, inner sep = 0,  fill=black!}]
		\node[label=below:{\footnotesize{$r'$}}] (r) at (0,0) {};
		\node[label=below:{\footnotesize{$u'$}}] (u1) at (-1,1)  {};
		
		\node[label=left:{\footnotesize{$v'$}}] (v1) at (1,1)  {};
		
		\foreach \from/\to in {r/u1, r/v1}
		\draw[->] (\from) -- (\to);

		\end{tikzpicture}
		
		\caption{\small{A branching with roots $r$ and $r'$. In this example, $u_1 \gtrdot r$, i.e.\ $u_1$ is a parent of $r$, but $u_2$ is not a parent of $r$ even though $u_2 > r$ ($u_2$ is a ``grandparent'' of $r$), and $v'\not>r$. A linear extension of this (or a leaf-to-root ordering) might present the vertices $u', v'$ and $r'$ before $u_3$, so it is not necessary that the roots of the branching are the last vertices presented.}}\label{fig:branching}
	}
\end{figure}

Denote the roots of the branching before a new vertex $u$ is presented by $v_1,v_2,\ldots$ indexed in the order in which they were presented.
We say that a \emph{leaf-to-root} ordering is \emph{geometric} if the parents of $u$ form an interval in this order, i.e., for every $u$, $\{v_i\mid u \lessdot v_i\}=\{v_i \mid l<i<r\}$ for some $l$ and $r$.
Intuitively, we do not allow an ordering of the following type.
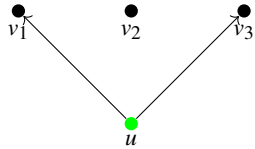
\begin{figure}[H]
		\centering{
				\begin{tikzpicture}
				[scale = 0.5, every node/.style={circle, minimum size = 5pt, inner sep = 0,  fill=black!}]
				\node[label=below:{\footnotesize{$v_1$}}] (p) at (1,4) {};
				\node[label=below:{\footnotesize{$v_2$}}] (q) at (4,4) {};
				\node[label=below:{\footnotesize{$v_3$}}] (r) at (7,4) {};
				\node[label=below:{\footnotesize{$u$}}, color=green] (s) at (4,1) {};
				
				\foreach \from/\to in {s/r, s/p}
				\draw[->] (\from) -- (\to);
				
				\end{tikzpicture}
				\caption{\small{The order $v_1, v_2, v_3, u$ is a leaf-to-root order but it is not \textit{geometric}, because $u$ is adjacent to $v_1$ and $v_3$ but not to $v_2$.}}
		}
		\end{figure}

Now we state a stronger form of \autoref{thm:arb1}.
\begin{theorem}\label{thm:arb2}
For any numbers $m,k$, there is no semi-online $k$-colouring algorithm that receives the vertices of an arborescence in a geometric leaf-to-root order and maintains that at every stage, all directed paths of length $m$ contain all $2$ colours. 
\end{theorem}
Call a semi-online $k$-colouring algorithm \emph{$m$-proper} if any path on $m$ vertices contains at least two colours. 
The theorem states there is no $m$-proper semi-online $k$-colouring algorithm for arborescences presented in geometric leaf-to-root order.
The idea of the proof is that for any vertex $u$, there are only finitely many possibilities for all directed paths of length $m$ from $u$.
However, we can always force the algorithm to produce a new ``type'' of path, leading to a contradiction.

Fix $k$ colours, $C_1, \dots, C_k$, a branching \F with a geometric leaf-to-root order, a point $p \in V(\F)$, and the time $t$ at which $p$ appears.
To ease future notation, let us get some (many) definitions out of the way.
\begin{itemize}
\item $p_u$ is a \textit{$u$-parent} of $p$ if there is a directed path $(p,p_1, \dots, p_u)$, i.e.\ $dist(p,p_u)=u$ in the graph.
We refer to the subpath $(p_1, \dots, p_u)$ as the \emph{chain} corresponding to $p_u$.
\item A $u$-parent $p_u$ of $p$ is \emph{in $C_i$} if $p_u$ is a $u$-parent of $p$ and every point in the chain $(p_1, \dots, p_u)$ is coloured with $C_i$ at time $t$. (Note that $p$ itself need not have colour $C_i$.)
\item A $u$-parent $p_u$ of $p$ in $C_i$ is \textit{maximal} if there is no $p_{u+1} \gtrdot p_u$ that is also coloured with $C_i$ at time $t$ (note that this depends \textit{only} on $t$, even if some such $p_{u+1}$ is coloured later).
\item Similarly, $p_u$ is an \textit{uncoloured} $u$-parent of $p$ if every point of $(p_1, \dots, p_u)$ is uncoloured, and it is a \textit{maximal} uncoloured $u$-parent if there is no $p_{u+1} \gtrdot p_u$ that is also uncoloured.
\item 
The \textit{type} of $p$, $tp(p)$ is defined as the vector $(t_1, \dots, t_k) \in \mathbb{N}^k$, where
\[t_i = \max \{u: p \text{ has a maximal } u\text{-parent in } C_i\}\]
\item If two partially coloured trees, $\T_1$ and $\T_2$, are isomorphic, we write $\T_1 \cong \T_2$.
Note that for the isomorphism we require that vertices coloured, say red, must be mapped to red vertices - we do not allow the isomorphism to permute the colours.
\end{itemize} 

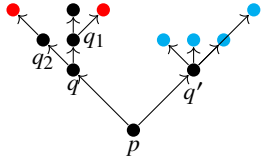
\begin{figure}[H]
	\vspace{0.2cm}
	\centering{
		\begin{tikzpicture}
		[auto=left,every node/.style={circle, minimum size = 5pt, inner sep = 0,  fill=black!},scale=0.8]
		\node[label=below:{\footnotesize{$p$}}] (p) at (0,0) {};
		
		\node[label=below:{\footnotesize{$q$}}] (q) at (-1,1) {};
		
		\node[label=below:{\footnotesize{$q_2$}}] (q2) at (-1.5,1.5) {};
		\node[label=below:{\footnotesize{}}, fill=red] (rq2) at (-2,2) {};
		
		\node[label=right:{\footnotesize{$q_1$}}] (q1) at (-1, 1.5) {};
		\node[label=below:{\footnotesize{}}] (d) at (-1,2) {};
		\node[label=below:{\footnotesize{}}, fill=red] (rq1) at (-0.5,2) {};
		
		\node[label=below:{\footnotesize{$q'$}}] (q') at (1,1) {};
		
		\node[label=below:{\footnotesize{}}, fill=cyan] (b1) at (0.5,1.5) {};
		\node[label=below:{\footnotesize{}}, fill=cyan] (b2) at (1,1.5) {};
		\node[fill=cyan] (2b1) at (1.5,1.5) {};
		\node[label=below:{\footnotesize{}}, fill=cyan] (2b2) at (2,2) {};

		\foreach \from/\to in {p/q, q/q1, q/q2, q1/d, q1/rq1, q2/rq2, p/q', q'/b1, q'/b2, q'/2b1, q'/2b2}
		\draw[->] (\from) -- (\to);

		\end{tikzpicture}
		\caption{\small{Suppose this is the partially $2$-coloured arborescence when $p$ appears. $q$ is an uncoloured $1$-parent of $p$, while $q_2$ is a \textit{maximal} uncoloured $2$-parent.
		The type of $p$ is $(0,0)$. We cannot say anything about the type of $q'$, because this depends on the time at which $q'$ appears.
		If the blue points were coloured \textit{when} $q'$ appeared, then $tp(q') = (0,2)$. If they were only coloured when $p$ appeared, then $tp(q') = (0,0)$, even though it now has blue points above it.}}
		}
	\end{figure}	

Let $S_t$ be the set of points that have appeared by time $t$ in the same connected component of $\mathcal{F}$ as $p$ (or in the subtree rooted at $p$ at time $t$). We now associate to $p$ a tree $\T(p)$ by ``trimming'' the induced subgraph $\mathcal{F}[S_t]$ in the following steps. (See \autoref{fig:trimming}.)

\begin{enumerate}
	\item If $q$ is uncoloured and $dist(p,q) > m$, delete $q$.
	\item If $q_1$ and $q_2$ are both maximal $t_i$-parents in $C_i$ for some remaining $q$, delete $q_2$ and all points that are $>q_2$.
	\item For $i=1,\ldots m$, if $q$ is a $(m-i)$-parent of $p$, and $q_1 \gtrdot q$ and $q_2 \gtrdot q$ are such that the subtrees rooted at $q_1$ and $q_2$ are isomorphic, delete $q_2$.
\end{enumerate}

\begin{figure}[H]
	\vspace{0.2cm}
	\centering{
		\begin{tikzpicture}
		[scale=0.7,auto=left,every node/.style={circle, minimum size = 5pt, inner sep = 0,  fill=black!}]
		\node[label=below:{\footnotesize{$p$}}] (p) at (0,0) {};
		
		\node[label=below:{\footnotesize{$q$}}] (q) at (-1,1) {};
		
		\node[label=below:{\footnotesize{$q_2$}}] (q2) at (-1.5,1.5) {};
		\node[label=below:{\footnotesize{}}, fill=red] (rq2) at (-2,2) {};
		
		\node[label=right:{\footnotesize{$q_1$}}] (q1) at (-1, 1.5) {};
		\node[label=below:{\footnotesize{}}] (d) at (-1,2) {};
		\node[label=below:{\footnotesize{}}, fill=red] (rq1) at (-0.5,2) {};
		
		\node[label=below:{\footnotesize{$q'$}}] (q') at (1,1) {};
		
		\node[label=below:{\footnotesize{}}, fill=cyan] (b1) at (0.5,1.5) {};
		\node[label=below:{\footnotesize{}}, fill=cyan] (b2) at (1,1.5) {};
		\node[fill=cyan] (2b1) at (1.5,1.5) {};
		\node[label=below:{\footnotesize{}}, fill=cyan] (2b2) at (2,2) {};

		\foreach \from/\to in {p/q, q/q1, q/q2, q1/d, q1/rq1, q2/rq2, p/q', q'/b1, q'/b2, q'/2b1, q'/2b2}
		\draw[->] (\from) -- (\to);

		\end{tikzpicture}
		\hspace{0.2cm}
		\begin{tikzpicture}
		[scale=0.7,auto=left,every node/.style={circle, minimum size = 5pt, inner sep = 0,  fill=black!}]
		\node[label=below:{\footnotesize{$\rightarrow$}} ,fill=none] at (-3,1) {};
		\node[label=below:{\footnotesize{step 1}} ,fill=none] at (-3,1) {};
		\node[label=below:{\footnotesize{$p$}}] (p) at (0,0) {};
		\node[label=below:{\footnotesize{$q$}}] (q) at (-1,1) {};
		
		\node[label=below:{\footnotesize{$q_2$}}] (q2) at (-1.5,1.5) {};
		\node[label=below:{\footnotesize{}}, fill=red] (rq2) at (-2,2) {};
		
		\node[label=right:{\footnotesize{$q_1$}}] (q1) at (-1, 1.5) {};
		
		\node[label=below:{\footnotesize{}}, fill=red] (rq1) at (-0.5,2) {};
		
		\node[label=below:{\footnotesize{$q'$}}] (q') at (1,1) {};
		
		\node[label=below:{\footnotesize{}}, fill=cyan] (b1) at (0.5,1.5) {};
		\node[label=below:{\footnotesize{}}, fill=cyan] (b2) at (1,1.5) {};
		\node[fill=cyan] (2b1) at (1.5,1.5) {};
		\node[label=below:{\footnotesize{}}, fill=cyan] (2b2) at (2,2) {};

		\foreach \from/\to in {p/q, q/q1, q/q2, q1/rq1, q2/rq2, p/q', q'/b1, q'/b2, q'/2b1, q'/2b2}
		\draw[->] (\from) -- (\to);
		
		\end{tikzpicture}
		\hspace{0.2cm}
		\begin{tikzpicture}
		[scale=0.7,auto=left,every node/.style={circle, minimum size = 5pt, inner sep = 0,  fill=black!}]
		\node[label=below:{\footnotesize{$\rightarrow$}} ,fill=none] (p) at (-3,1) {};
		\node[label=below:{\footnotesize{step 2}} ,fill=none] at (-3,1) {};
		\node[label=below:{\footnotesize{$p$}}] (p) at (0,0) {};
		\node[label=below:{\footnotesize{$q$}}] (q) at (-1,1) {};
		
		\node[label=below:{\footnotesize{$q_2$}}] (q2) at (-1.5,1.5) {};
		\node[label=below:{\footnotesize{}}, fill=red] (rq2) at (-2,2) {};
		
		\node[label=right:{\footnotesize{$q_1$}}] (q1) at (-1, 1.5) {};
		
		\node[label=below:{\footnotesize{}}, fill=red] (rq1) at (-0.5,2) {};
		
		\node[label=below:{\footnotesize{$q'$}}] (q') at (1,1) {};
		
		\node[label=below:{\footnotesize{}}, fill=cyan] (b1) at (0.5,1.5) {};
		\node[fill=cyan] (2b1) at (1.5,1.5) {};
		\node[label=below:{\footnotesize{}}, fill=cyan] (2b2) at (2,2) {};

		\foreach \from/\to in {p/q, q/q1, q/q2, q1/rq1, q2/rq2, p/q', q'/b1, q'/2b1, q'/2b2}
		\draw[->] (\from) -- (\to);
		
		\end{tikzpicture}
		\hspace{0.2cm}
		\begin{tikzpicture}
		[scale=0.7,auto=left,every node/.style={circle, minimum size = 5pt, inner sep = 0,  fill=black!}]
		\node[label=below:{\footnotesize{$\rightarrow$}} ,fill=none] () at (-3,1) {};
		\node[label=below:{\footnotesize{step 3}} ,fill=none] at (-3,1) {};
		\node[label=below:{\footnotesize{$p$}}] (p) at (0,0) {};
		\node[label=below:{\footnotesize{$q$}}] (q) at (-1,1) {};

		\node[label=right:{\footnotesize{$q_1$}}] (q1) at (-1, 1.5) {};
		
		\node[label=below:{\footnotesize{}}, fill=red] (rq1) at (-0.5,2) {};
		
		\node[label=below:{\footnotesize{$q'$}}] (q') at (1,1) {};
		
		\node[label=below:{\footnotesize{}}, fill=cyan] (b1) at (0.5,1.5) {};
		\node[fill=cyan] (2b1) at (1.5,1.5) {};
		\node[label=below:{\footnotesize{}}, fill=cyan] (2b2) at (2,2) {};

		\foreach \from/\to in {p/q, q/q1, q1/rq1,  p/q', q'/b1, q'/2b1, q'/2b2}
		\draw[->] (\from) -- (\to);
		
		\end{tikzpicture}
		
		\caption{\small{Example for trimming with $m=k=2$. In step 1, we delete the uncoloured $3$-parent of $p$, but preserve the red parent of $q_1$. In step 2, we ``trim'' the blue parents of $q'$. In step $3$, the subtrees rooted at $q_1$ and $q_2$ are isomorphic, so we delete $q_2$.}}\label{fig:trimming}
	}
	\vspace{0.2cm}
\end{figure}
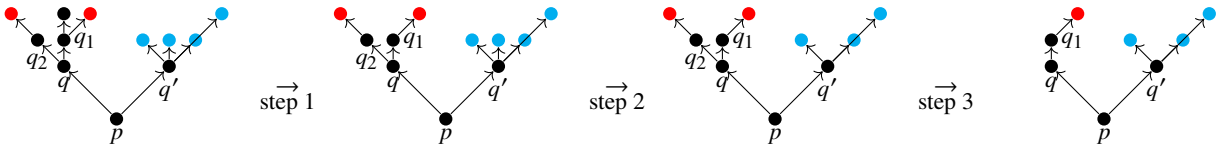 

The idea of this trimming process is to retain only the ``essential'' information about the colouring when $p$ appears and reduce the number of possible $\T(p)$ to a bounded number of options.
If we assume that the algorithm has produced a $m$-proper colouring until the time that $p$ appears, then we can disregard vertices at distance $>m$ from $p$.
If a vertex was not deleted during the trimming, we say that it was \emph{preserved}.

We could modify step $1$ to delete \textit{all} points at distance $> m$ from $p$. However, in the proof we will use the fact that the type of any point at distance $\leq m$ from $p$ is preserved. Of course, if the algorithm is good, then any directed path of length $m$ contains at least $2$ colours, so deleting only the uncoloured points is just a technical condition that simplifies notation.
Finally, in step $3$, we ensure that we do not have any ``repetitions''. For example, if all the branches rooted at $p$ are isomorphic, by considering only one of them we do not lose any important information.

We emphasise that $\T(p)$ depends only on the time at which $p$ appears. For instance, in the above figure, even if $q'$ is coloured blue at a later time, $\T(p)$ does not change. 

\subsection{Proof of the main theorem}

The crucial result of the trimming process is that the following lemma holds.
\begin{lemma}\label{lem:trim}
	Suppose that a semi-online colouring algorithm as in the statement of the theorem exists. Then the following hold.
	
	\begin{enumerate}
		
		\item The set
		$\{\T(p): \mathcal{F} \text{ is a branching}, p \in V(\mathcal{F})\}$ 
		is finite.
		\item  If $q \in S_t$ is preserved after the trimming, and $q$ had an $t_i$-parent in $C_i$ in $\mathcal{F}$, then $q$ has an $t_i$-parent in $C_i$ in $\T(p)$. In particular, the type of $q$ is preserved.
		\item Suppose $p' \lessdot p$ is presented, and $q$ was an uncoloured $u$-parent of $p$ in $\T(p)$ for $u<d$. If none of the points on the chain from $p'$ to $q$ are coloured when $p'$ is presented, then $q$ is preserved in $\T(p')$. 
	\end{enumerate}
\end{lemma}

\begin{proof}
    We show that there are only finitely many possibilities for $\T(p)$. In step $1$ of the trimming we delete uncoloured points at distance $>m$ from $p$. In step $2$ we preserve only maximal parents in $C_i$ of $p$ for each colour $C_i$. Since the algorithm is $m$-proper, $\T(p)$ will have depth at most $m$. In step $2$, we also delete ``repetitions'' so there are only finitely many possibilities for each of the branches above $p$. And in step $3$, we delete isomorphic subtrees, so no two of the branches above $p$ are isomorphic. Thus $\T(p)$ can take only finitely many values.
	
	The second claim follows from our earlier argument.
	For the third claim, we only need to consider the case when the algorithm produces an uncoloured $u$-parent $q'$ such that one of $q$ and $q'$ must be trimmed (i.e., the subtrees rooted at $q$ and $q'$ are isomorphic). In this case, we can  assume without loss of generality that $q'$ is trimmed so the second property holds.
\end{proof}

\begin{lemma}\label{lem:descent}
	At any stage of the algorithm, suppose that we have a collection of trees with roots $p_1,\ldots,p_s$ presented in this order such that no two $\T(p_i)$ and $\T(p_j)$ are isomorphic.
	Then presenting a vertex $p$ with parents $p_1,\ldots,p_s$, will give a tree $\T(p)$ that is non-isomorphic to any $\T(p_i)$.
\end{lemma}

\begin{proof}
	Suppose for contradiction that for some $p_i$, $\T(p) \cong \T(p_i)$.
	Let $\varphi: \T(p) \to \T(p_i)$ be an isomorphism (preserving colourings).
	We prove by induction for all $u < m$ that there is a chain $p=r_0 \lessdot r_1 \lessdot \dots \lessdot r_u$ in $\T(p)$ such that for all $i\le u$ we have $\varphi(r_{i-1}) = r_i$, and $r_i$ is uncoloured.
	
	First suppose $p$ is coloured, say with $C_1$, in $\T(p)$, and let $t_1$ be maximal such that $p$ has a $t_1$-parent in $C_1$. 
	$r_1 = p_i= \varphi(p)$ was coloured with $C_1$ in $\T(r_1)$, and by the isomorphism $r_1$ has a $t_1$-parent in $C_1$. 
	Since we did not recolour any points, this produces a $(t_1+1)$-parent in $C_1$ of $p$ in $\T(p)$, contradicting the maximality of $t_1$.
	
	So $p$ must be uncoloured in $\T(p)$, which implies that $r_1$ was uncoloured in $\T(r_1)$. 
	To complete the base case of the induction hypothesis, we need to show that $r_1$ remains uncoloured in $\T(p)$, i.e., when $p$ appears. Let $t_1$ be as earlier, and suppose again  that $r_1$ is coloured with $C_1$ in $\T(p)$. 
	Then $p$ has an $(t_1+1)$-parent in $C_1$ in $\T(p)$, again a contradiction.
	
	Suppose we have produced a chain $p=r_0 \lessdot r_1 \lessdot \dots \lessdot r_{u-1}$ from the induction hypothesis. 
	If $u-1 = m$, then we are done.
	Otherwise, let $r_u = \varphi(r_{u-1})$. 
	Then $r_u$ is uncoloured in $\T(r_1)$. 
	Since $r_{u-1} \gtrdot r_{u-2}$, $r_u \gtrdot \varphi(r_{u-2}) = r_{u-1}$, so $p=r_0 \lessdot r_1 \lessdot \dots \lessdot r_u$ is a chain, and it remains to show that $r_u$ is uncoloured in $\T(p)$. 
	Suppose $r_u$ is coloured in $\T(p)$ with $C_1$. 
	If $s_1$ is maximal so that $r_{u-1}$ has an $s_1$-parent in $C_1$ in $\T(p)$, then $r_u$ has an $s_1$-parent in $C_1$ in $\T(p_1)$, producing an $(s_1+1)$-parent in $C_1$ for $r_{u-1}$ in $\T(p)$.
	This contradicts the maximality of $s_1$.
	
 This eventually produces a chain of $m$ uncoloured points, which contradicts the correctness of the semi-online algorithm.
\end{proof}
\begin{proof}
From here we can finish the proof of \autoref{thm:arb2} with an infinite descent argument as follows.
Order the finite sequences of naturals, $\N^{<\omega}$, such that $(s_1,s_2,\ldots,s_l)>(s_1',s_2',\ldots,s_{l'}')$ if there is some $i$ such that for all $j<i$ we have $s_j=s_j'$ but $s_i>s_i'$, or $l>l'$ and for all $j\le l'$ we have $s_j=s_j'$.
For a branching \F, we define its \emph{associated sequence} as follows.
For each root $p_i$ of \F, consider the sequence of trees $\T(p_i)$ in the order their roots were presented.
Let $i_1$ be the smallest index such that for every $\T(p_i)$ there is an $i'\le i_1$ such that $\T(p_i)\cong\T(p_{i'})$.
The number of different trees $\T(p_i)$ (same as the number of different trees up to $i_1$) is denoted by $s_1$.
In general, after $i_{j-1}$ has been defined, let $i_j$ be the smallest index such that for every $\T(p_i)$ with $i>i_{j-1}$  there is an $i_{j-1}<i'\le i_j$ such that $\T(p_i)\cong\T(p_{i'})$.
The number of different trees $\T(p_i)$ for $i_{j-1}<i\le i_{j}$ is denoted by $s_j$.
We repeat this for $N$ steps, where $N$ denotes the number of possible different (i.e., non-isomorphic) trees \T, or until there are no more roots in \F.
The numbers $(s_1,\ldots,s_l)$ are the associated sequence of \F.

Note that there are finitely many associated sequences, as each $N\ge s_1\ge s_2\ge \dots \ge s_l$, and also $l\le N$.
Applying \autoref{lem:descent} to the largest associated sequence that can be attained during the run of the semi-online algorithm, we get a contradiction as follows.
Let \F be a branching whose associated sequence, $(s_1,\ldots,s_l)$, is the largest.

Case 1: If $s_1=N$, then we present a new point $p$ whose parents are the roots of $\F$, and by \autoref{lem:descent} we produce a new tree, which is not possible.

Case 2: If $N> s_1> \dots> s_l$, then $l<N$. Introduce a new vertex disjoint from all vertices of $\F$. This will either increase an earlier $s_i$, or give a new $s_{l+1}=1$, but both of these contradict the maximality of $(s_1,\ldots,s_l)$. 

Case 3: There is some $j$ for which $s_j=s_{j+1}$. This is only possible if all the trees $\T(p_i)$ for $i_{j-1}<i\le i_{j}$ have an isomorphic copy $\T(p_{i'})$ for some $i_{j}<i\le i_{j+1}$.
Introduce a new vertex $p$ under all the roots $p_i$ of $\F$ with index $i>i_{j}$ to obtain a new branching $\F'$.
By \autoref{lem:descent}, the tree $\T(p)$ is non-isomorphic to any $\T(p_i)$ with $i_{j-1}<i\le i_{j}$.
Therefore, the associated sequence of $\F'$ will be larger than $(s_1,\ldots,s_l)$, contradicting its maximality.

In summary, it is not possible for a semi-online $k$-colouring algorithm to produce finitely many associated sequences, so it cannot be $m$-proper.
\end{proof}

\subsection{Application to bottomless rectangles}

In this section, we apply \autoref{thm:arb2} to semi-online colouring algorithms for Erd{\H o}s-Szekeres configurations.
We start with towers.

\begin{corollary}\label{cor:tower}
	There is no semi-online colouring algorithm for towers from above, i.e., for any numbers $k$ and $m$, for any semi-online algorithm that $k$-colours bottomless rectangles from above, there is a family of bottomless rectangles such that any two intersecting rectangles form a tower, and the algorithm produces an $m$-fold covered point that is covered by at most one colour.
\end{corollary}
\begin{proof} In order to apply \autoref{thm:arb2}, we need to show that any branching can be realised as a family of towers so that
	\begin{enumerate}
		\item ordering the rectangles from above corresponds to a geometric leaf-to-root order of the branching, and
		\item a semi-online colouring algorithm for towers from above corresponds to an $m$-proper semi-online $k$-colouring algorithm for branchings in this order.
	\end{enumerate}
	
	For any arborescence $\mathcal{F}$ in geometric leaf-to-root order, we show by induction on $|\mathcal{F}|$ that it can be realised as a family of towers with this order. For $|\mathcal{F}| = 1$ this is clear. For the inductive step, we will need to use the fact that the ordering is geometric. For example, suppose we have a non-geometric order and three roots $p,q,r$ that are realised as disjoint rectangles, with $q$ between $p$ and $r$. Then if the next root $s$ is presented with $s<p$ and $s<r$, but $s \nless q$, $s$ cannot be realised as a rectangle.
	
	\begin{figure}[H]
		\centering
		\begin{tikzpicture}[scale=0.7]
		\draw (0,0)--(0,1)--(1,1)--(1,0);
		\draw (2,0)--(2,1)--(3,1)--(3,0);
		\draw (4,0)--(4,1)--(5,1)--(5,0);
		\node[left] at (0,0) {$p$};
		\node[left] at (2,0) {$q$};
		\node[left] at (4,0) {$r$};
		
		\end{tikzpicture}
		\caption{\small{There is no way to present a new rectangle $s$ that intersects $p$ and $r$ but not $q$.}}
	\end{figure}
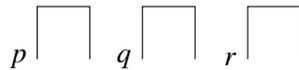

	Now we prove the induction step. Let $|\mathcal{F}| = n$, and $r$ be the last element in the ordering of $V(\mathcal{F})$. Take any realisation of $\mathcal{F} \setminus \{r\}$ as a family of towers. If $r$ is an isolated vertex in $\mathcal{F}$, present $r$ as a disjoint rectangle to the right of the realisation $\mathcal{F} \setminus \{r\}$. Otherwise, since the order is geometric, $r$ will only intersect some geometrically adjacent rectangles of $\mathcal{F}$ (by construction). Hence $r$ can be realised as a minimal rectangle.
	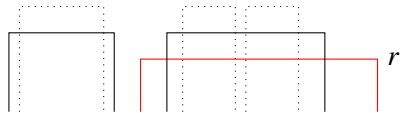
\begin{figure}[H]
		\centering{
			\begin{tikzpicture}[scale=0.7]
			
			\draw[color=red] (3.5,0)--(3.5,1)--(8,1)--(8,0);
			\draw (1,0)--(1,1.5)--(3,1.5)--(3,0);
			\draw (4,0)--(4,1.5)--(7,1.5)--(7,0);
			\draw[dotted] (4.3,0)--(4.3,2)--(5.3,2)--(5.3,0);
			\draw[dotted] (5.5,0)--(5.5,2)--(6.5,2)--(6.5,0);
			\draw[dotted] (1.2,0)--(1.2,2)--(2.8,2)--(2.8,0);
			
			\node[right] at (8,1) {$r$};
			
			\end{tikzpicture}
			\caption{\small{By the geometric ordering, we can realise $r$ as a minimal element.}}
		}
		
	\end{figure}
The proof for nested rectangles from below is analogous.
\end{proof}

\begin{corollary}
	There is no semi-online $k$-colouring algorithm from the left or from below for increasing steps. More precisely, for any integers $k$ and $m$, there is no semi-online algorithm to $k$-colour rectangles from the left (or from below) so that at every step, any point covered by $m$-increasing steps is covered by at least $2$ colours.
	Similarly, there is no semi-online colouring algorithm for decreasing steps from the right or from below.
\end{corollary}

Note that this statement is slightly weaker than \autoref{thm:rect} or \autoref{cor:tower} because we do not exclude the other kind of configurations from the family.

\begin{proof}
	
	We first prove the statement for increasing steps from the left. Again, we will prove the corollary by induction on $|\mathcal{F}|$. However, we also weaken our requirements for the colouring of the steps. 
	That is, we need not assume that every directed path of length $m$ in the branching corresponds to a point covered by exactly $m$ increasing steps. 
	It is easy to see that a semi-online colouring algorithm of $\mathcal{F}$ is $m$-proper if and only if when any point $p \in \mathcal{F}$ is presented, any directed path of length $m$ \textit{from} $p$ contains at least $2$ colours.  
	So it suffices to prove the following by induction.
	
	Any branching $\mathcal{F}$ with a geometric leaf-to-root order can be realised as a family of bottomless rectangles so that
	\begin{enumerate}
		\item when $p \in \mathcal{F}$ is presented, we realise $p$ as a rectangle so that any directed path of length $m$ from $p$ corresponds to a point covered by exactly $m$ increasing steps, and
		\item any two rectangles intersect either as increasing or as decreasing steps.
	\end{enumerate}
	
	The second assumption is a technical condition to ensure that $q$ covers the top-right corner of $r$ if and only if $(r,q)$ form increasing steps, so we can choose the top-right corner of an appropriate rectangle as the point satisfying the first induction hypothesis.
	
	The case $|\mathcal{F}| = 1$ is trivial. Let $|\mathcal{F}| = n$, and $r$  be the last element in the ordering of $\mathcal{F}$. Take any realisation of $\mathcal{F} \setminus \{r\}$ satisfying the induction hypotheses. If $r$ is an isolated vertex, let $q$ be the last element presented (thus a root), and realise $r$ as a rectangle so that $(q,r)$ form decreasing $2$-steps (see \autoref{fig:stepsroot}). There are no directed paths of length $m$ from $r$ so both induction hypotheses are satisfied. 
	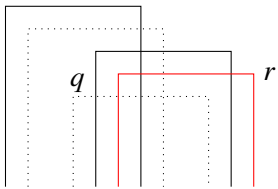
\begin{figure}[H]
		\centering{
			\begin{tikzpicture}[scale=0.6]
			\draw (0,0)--(0,4)--(3,4)--(3,0);
			\draw[dotted] (0.5,0)--(0.5,3.5)--(3.5,3.5)--(3.5,0);
			
			\draw[dotted] (1.5,0)--(1.5,2)--(4.5,2)--(4.5,0);
			\draw (2,0)--(2,3)--(5,3)--(5,0);
			\draw[color=red] (2.5,0)--(2.5,2.5)--(5.5,2.5)--(5.5,0);
			
			\node[right] at (5.5,2.5) {$r$};
			\node[left] at (2,2.3) {$q$};
			
			\end{tikzpicture}
		}
		\caption{\small{The rectangles in decreasing steps correspond to roots of the branching.}}\label{fig:stepsroot}
	\end{figure}

	Otherwise, since the ordering is geometric, $r$ will only intersect the rightmost rectangles (by construction), thus can be realised as a rectangle that forms increasing steps with these rightmost roots, and decreasing steps with the other roots (see \autoref{fig:stepsbranching}).
	
	To see that the first hypothesis is satisfied, consider the rectangles corresponding to any directed path of length $m$ from $r$, say $(r_1, \dots, r_{m-1}, r)$. Then the top-right corner of $r_1$ will not be covered by any rectangle other than the ones in this chain - this follows from the induction hypothesis and the fact that $\mathcal{F}$ is a branching, so $r_2$ is the unique child of $r_1$.
	\begin{figure}[H]
		\centering
		\begin{tikzpicture}[scale=0.6]
		\draw (0,0)--(0,4)--(3,4)--(3,0);
		\draw[dotted] (0.5,0)--(0.5,3.5)--(3.5,3.5)--(3.5,0);
		\draw (1,0)--(1,1.5)--(4,1.5)--(4,0);
		\draw[dotted] (1.5,0)--(1.5,2)--(4.5,2)--(4.5,0);
		\draw (2,0)--(2,2.5)--(5,2.5)--(5,0);
		\draw[color=red] (2.5,0)--(2.5,3)--(5.5,3)--(5.5,0);

		\node[right] at (5.5,3) {$r$};
		\node[left] at (2,2.5) {$r_{m-1}$};
		\node[left] at (1,1.5) {$r_1$};

		\end{tikzpicture}
		\caption{\small{The new rectangle $r$ is presented in increasing steps with $r_{m-1}$, and decreasing steps with the other minimal elements.}}\label{fig:stepsbranching}
	\end{figure}
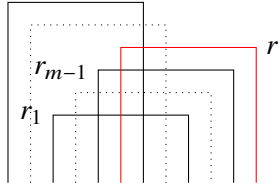
	
	The proof for increasing steps from below follows analogously, except we change the second induction hypothesis to assume that any two rectangles intersect either as increasing steps or as a tower. In this construction, the roots of the branching at any time will correspond to a tower, and the geometric ordering ensures that a new root can be placed in increasing steps with the top-most rectangles of the tower.
	
	The proof for decreasing steps is exactly the same, only interchanging left and right.
\end{proof}

\section{Further results for bottomless rectangles}\label{sec:further}

\subsection{No shallow hitting sets}

To prove an upper bound for intersecting families in \autoref{thm:intersecting}, we used the fact that if a family admits a $c$-shallow hitting set for arbitrary depths $d$, then $m_k^* \leq ck$.
Unfortunately, general families of bottomless rectangles do not admit such hitting sets.
\begin{theorem}\label{thm:noshallowhittingset}
For every integer $c \geq 0$, there exists a real number $r < 1$ and an integer $D \geq 1$ such that for every integer $d \geq D$, there is a family $\F = \F(h,d)$ such that the following holds.
For any hitting set $H$ of the $d$-cells of $\F$, there is a vertical line $\ell$ such that $|\ell \cap \F| \leq \lceil rd \rceil$ and $|\ell \cap H| \geq h$.
\end{theorem}
In particular, $\F (h,d)$ witnesses that there are no $(h-1)$-shallow hitting sets for depth $d$.
\begin{proof}
We construct $\F(h,d)$ by induction on $h$.
For $h=0$, letting $r=0$, $D = 1$, and $\F(h,d)$ be the empty family, the conditions are satisfied.
Let $r_{h-1}$ and $D_{h-1}$ be the values given by the induction hypothesis for $h-1$.
Choose $s$ to be a real number such that $0 < s < 1-r_{h-1}$.
Set
\[
D > \max \Big(\frac{D_{h-1}}{s}, \frac{1}{s(1-r_{h-1})}\Big); \quad
r = \max\Big(r_{h-1}+s, 1-s(1-r_{h-1}) - \frac{1}{D}\Big).
\]
For $d \geq D$, we define $\F = \F(h,d)$ as follows.
Take $d$ rectangles $(R_1, \dots, R_d)$ in increasing $d$-steps.
We will insert families of the form $\F(h-1, i)$ for some $D_{h-1} \leq i \leq d$.
First, set $i_0 = \lfloor sd \rfloor$, so that $i_0 \geq \lfloor sD \rfloor D_{h-1} \geq 1$.
Insert a copy of $\F(h-1,d)$ so that it intersects only the rectangles $R_1, \dots , R_{i_0}$, and the top sides of all the rectangles are above that of $R_{i_0}$.
Next, for every $i_0 < i \leq d$, insert a copy of $\F(h-1,i)$ so that it intersects exactly the rectangles $R_i, \dots, R_d$, and all the top sides lie between those of $R_i$ and $R_{i+1}$ (or just above that of $R_d$ when $i=d$).
\begin{figure}[H]
    \centering{
    \begin{tikzpicture}[y=0.7cm, x=1cm]
    \draw[color=gray] (0,0)--(0,1)--(6,1)--(6,0);
    \node[below] at (0,0) {\footnotesize $R_1$ };
    \draw[dotted] (0.5,0)--(0.5,1.5)--(6.5,1.5)--(6.5,0);
    \draw[color=gray] (1,0)--(1,2)--(7,2)--(7,0);
    \node[below] at (1,0) {\footnotesize $R_{i_0}$ };
    \draw[color=red] (1.5,0)--(1.5,2.5)--(2,2.5)--(2,0);
    \node[above] at (1.75,2.5) {\footnotesize $\F(h-1, d)$};
    
    \draw[color=gray] (3,0)--(3,3)--(9,3)--(9,0);
    \node[below] at (3,0) {\footnotesize $R_{i_0 + 1}$};
    \draw[color=red] (7.5,0)--(7.5,3.5)--(8,3.5)--(8,0);
    \node[above] at (7.25,3.5) {\footnotesize $\F(h-1, i_0+1)$};

    \draw[color=gray] (4,0)--(4,4)--(10.25,4)--(10.25,0);
    \node[below] at (4,0) {\footnotesize $R_{i_0 + 2}$};
    \draw[color=red] (9.25,0)--(9.25,4.25)--(9.75,4.25)--(9.75,0);
    \node[above] at (9.5,4.25) {\footnotesize $\F(h-1, i_0+2)$};
    
    \draw[dotted] (4.5,0)--(4.5,4.5)--(11,4.5)--(11,0);
    
    \draw[color=gray] (5,0)--(5,5)--(11.5,5)--(11.5,0);
    \node[below] at (5,0) {\footnotesize $R_d$ };
    \draw[color=red] (11.25,0)--(11.25,5.25)--(11.75,5.25)--(11.75,0);
    \node[above] at (11.5,5.25) {\footnotesize $\F(h-1, d)$};
    
    \node[label=below:{\footnotesize $p$ }, circle, fill=black,inner sep=0pt,minimum size=3pt] (p) at (5.5,0.5) {};

    \end{tikzpicture}
    \caption{\small{The construction of $\F = \F(h,d)$.}}
    }
    \label{fig:nohitting}
\end{figure}
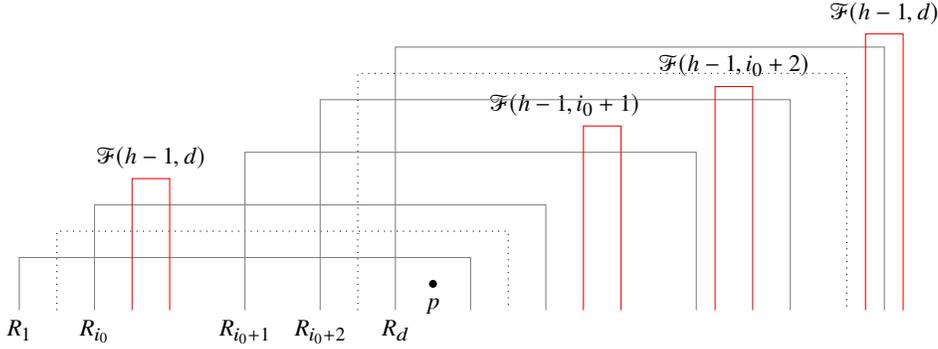
Now, let $H$ be a hitting set for the $d$-cells of $\F$.
Since the point $p$ is in a $d$-cell, $H$ must contain at least one of $R_1, \dots, R_d$.
First, suppose that $R_i \in H$ for some $i \leq i_0$.
The intersection of $H$ with the leftmost copy of $\F(h-1,d)$ is a hitting set for the $d$-cells in $\F(h-1,d)$.
By induction, there is a line $\ell$ such that $|\ell \cap \F(h-1,d)| \leq r_{h-1}d$ and $| \ell \cap H \cap \F(h-1,d)| \geq h-1$.
Then,
\[
|\ell \cap \F| \leq r_{h-1}d + i_0 \leq (r_{h-1}+s)d < rd
\]
and
\[
|\ell \cap H | \geq (h-1) + 1 \geq h.
\]
On the other hand, suppose $H$ does not contain any rectangle $R_i$ for $i \leq i_0$.
Let $i$ be maximal so that $R_i \in H$.
Then, $H \cap \F(h-1, i)$ is a hitting set for the $i$-cells in $\F(h-1,i)$.
Again, we have a vertical line $\ell$ such that $|\ell \cap \F(h-1,i)| \leq r_{h-1}i$ and $|\ell \cap H \cap \F(h-1,i)| \geq h-1$.
Similarly,
\[
|\ell \cap H| \geq h
\]
and
\begin{align*}
|\ell \cap \F| &\leq r_{h-1}i + (d-i) + 1\\
&= d+1 - i(1-r_{h-1})\\
&\leq d+1 - sd(1-r_{h-1})\\
&= d(1-s(1-r_{h-1}) + 1/d)\\
&\leq rd.
\end{align*}
\end{proof}

\subsection{An improved lower bound}

Finally, we present an improved lower bound for general bottomless rectangle families, and a weaker lower bound that can be applied to the steps problem.

\begin{theorem}\label{thm:lb}
	$m_k^*(\Frect) \geq 2k-1$ for general families of bottomless rectangles.
\end{theorem}
\begin{proof}
	Our lower bound construction proceeds in two steps. 
	
	\begin{enumerate}
		\item If $m_k^* < m_{k-1}^*+2$, then every family has a polychromatic $k$-colouring that is proper.
		\item There is a family so that no polychromatic $k$-colouring is proper.
	\end{enumerate}
	This contradiction shows that $m_k^* \geq m_{k-1}^* + 2$, so by induction $m_k^* \geq 2k-1$.
	
1. Suppose for some family $\mathcal{F}$, no polychromatic $k$-colouring of $\mathcal{F}$ is  proper. Let $\mathcal{G}$ be a witness to the sharpness of $m_{k-1}^*$, i.e.\ any $(k-1)$-colouring of $\mathcal{G}$ produces a point covered by $m_{k-1}^*-1$ rectangles but not all $k$ colours. In a small interval around every $2$-covered point in $\mathcal{F}$, we place a thin copy of $\mathcal{G}$ (see \autoref{fig:prop}).
	
	Any polychromatic colouring of this new family $\mathcal{F}'$ must induce a polychromatic colouring of $\mathcal{F}$, so some copy of $\mathcal{G}$ is covered by $2$ rectangles of the same colour, say red.
	By hypothesis, any point in this copy of $\mathcal{G}$ covered by at least $m_{k}^*$ rectangles is covered by all $k$ colours. 
	Since every such point is covered by exactly two red rectangles from $\mathcal{F}$, recolouring every red rectangle in $\mathcal{G}$ blue cannot ruin this property. 
	However, this induces a $(k-1)$-colouring of $\mathcal{G}$ so that any point in $m_{k-1}^*-1$ rectangles is covered by all $k-1$ colours, a contradiction.
	
	So every family must have a polychromatic colouring that is proper.
	
	\begin{figure}[H]
		\centering{	
			\begin{tikzpicture}[y=0.7cm]
			\draw (0,0)--(0,2)--(3,2)--(3,0);

			\draw (1,0)--(1,1)--(4,1)--(4,0);

			\draw (1.8,0)--(1.8,3)--(5,3)--(5,0);
			
			\node[circle, fill=black,inner sep=0pt,minimum size=3pt] at (0.5,0.6) {};
			\node[circle, fill=black,inner sep=0pt,minimum size=3pt] at (2.5,1.5) {};
			\node[circle, fill=black,inner sep=0pt,minimum size=3pt] at (3.5,0.4) {};
			
			\draw (0.3,0)--(0.3,1.2)--(0.7,1.2)--(0.7,0);
			\draw (2.3,0)--(2.3,1.7)--(2.7,1.7)--(2.7,0);
			\draw (3.3,0)--(3.3,0.6)--(3.7,0.6)--(3.7,0);

			\node[above] at (0.3,1.2) {$\mathcal{G}$};
			\node[left] at (2.3,1.7) {$\mathcal{G}$};
			\node[above] at (3.3,0.6) {$\mathcal{G}$};
			
			\end{tikzpicture}
			\caption{\small{Place disjoint thin copies of $\mathcal{G}$ around every $2$-covered point in $\mathcal{F}$.}}\label{fig:prop}
		}
		
	\end{figure}
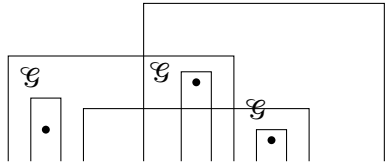
	
2. Consider the family in \autoref{fig:noprop}.
	
	\begin{figure}[H]
		\centering {
			\begin{tikzpicture}[x=0.5cm, y=0.4cm]
			\draw (0,0)--(0,1)--(8,1)--(8,0);
			\draw (1,0)--(1,2)--(7,2)--(7,0);
			\draw[dotted] (2,0)--(2,3)--(6,3)--(6,0);
			\draw (3,0)--(3,4)--(5,4)--(5,0);
			\draw (4,0)--(4,6)--(10,6)--(10,0);
			
			\node[circle, fill=black,inner sep=0pt,minimum size=3pt] at (7.5,0.5) {};
			\node[circle, fill=black,inner sep=0pt,minimum size=3pt] at (6.5,1.5) {};
			\node[circle, fill=black,inner sep=0pt,minimum size=3pt] at (4.5,3.5) {};
			\node[circle, fill=black,inner sep=0pt,minimum size=3pt,label=below:{\small$p$}] at (3.5,0.5) {};
			
			\node[right] at (10,6) {$R$};
			
			\end{tikzpicture}
			\caption{\small{No polychromatic colouring of this family will be proper.}}\label{fig:noprop}
		}
		
	\end{figure}
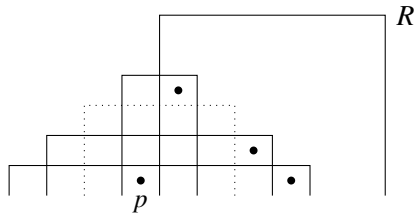
	
We have an $m$-tower (where $m$ may be arbitrarily large), so that each rectangle from the tower meets $R$ in a $2$-covered point. Suppose without loss of generality that $R$ is coloured red in some polychromatic $k$-colouring. For this colouring to be proper, no rectangle of the tower can be red - however the point $p$ will then be covered by $m$ rectangles, none of which are red, so the colouring cannot be polychromatic. This completes our proof. 
\end{proof}

This lower bound cannot be applied to \textit{unit bottomless}, as this construction relies heavily on towers.
For these, we prove the following weaker lower bound.

\begin{proposition}
	$m_k^*(\Funit) \geq 2\lfloor \frac{2k-1}{3}\rfloor + 1$ for \textit{unit bottomless}.
\end{proposition}
\begin{proof} This is a generalisation of the construction that shows that $m_k^* =3$. Let $\mathcal{F}$ be a family of $2k-1$ rectangles partitioned into $3$ almost equal subfamilies, $\mathcal{F}_1$, $\mathcal{F}_2$ and $\mathcal{F}_3$ as follows.
	
	\begin{figure}[H]
		\centering{	
			\begin{tikzpicture}[x=1cm, y=0.5cm]
			\draw (1,0)--(1,1)--(5,1)--(5,0);
			\draw[dotted](1.25,0)--(1.25,1.25)--(5.25,1.25)--(5.25,0);
			\draw(1.5,0)--(1.5,1.5)--(5.5,1.5)--(5.5,0);

			\draw (0,0)--(0,3)--(4,3)--(4,0);
			\draw[dotted](0.25,0)--(0.25,3.25)--(4.25,3.25)--(4.25,0);
			\draw(0.5,0)--(0.5,3.5)--(4.5,3.5)--(4.5,0);
			
			\draw (2.5,0)--(2.5,5.5)--(6.5,5.5)--(6.5,0);
			\draw[dotted] (2.75,0)--(2.75,5.75)--(6.75,5.75)--(6.75,0);
			\draw (3,0)--(3,6)--(7,6)--(7,0);
			
			\node[above] at (0.25,3.25) {$\mathcal{F}_1 $};
			
			\node[above] at (1.25,1.25) {$\mathcal{F}_2 $};
			
			\node[above] at (6.75,1.75) {$\mathcal{F}_3$};
			
			\node[circle, fill=black,inner sep=0pt,minimum size=3pt,label=below:{\small$p_2$}] at (3.5,2.25) {};
			
			\node[circle, fill=black,inner sep=0pt,minimum size=3pt,label=below:{\small$p_1$}] at (4.75,0.5) {};
			
			\node[circle, fill=black,inner sep=0pt,minimum size=3pt,label=below:{\small$p_3$}] at (2,0.5) {};
			
			\end{tikzpicture}
			\caption{This construction shows that $m_k \geq 2 \lfloor \frac{2k-1}{3}\rfloor +1$.}	
		}
		
	\end{figure}
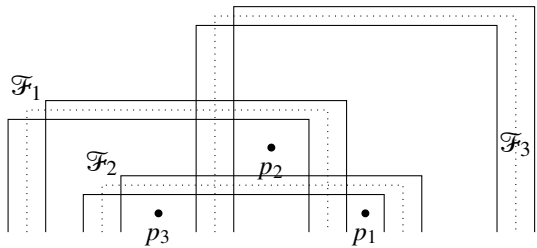
	
	Consider any $k$-colouring of $\mathcal{F}$. Some colour, say red, is used at most once, so it appears in at most one of $\mathcal{F}_1, \mathcal{F}_2$ and $\mathcal{F}_3$, say $\mathcal{F}_i$. Then the point $p_i$ is covered by the other two subfamilies, and no red rectangle.
	Since $\lfloor \frac{2k-1}{3}\rfloor \leq |\mathcal{F}_i| \leq \lceil \frac{2k-1}{3}\rceil$, this proves the lower bound.
\end{proof}

Note that the family in the figure does not contain any towers or nested sets. This gives a lower bound to complement \autoref{prop:stepspt}, namely that for steps, $m_k^* \geq  2 \lfloor \frac{2k-1}{3}\rfloor +1$.

\bibliographystyle{plain}

\end{document}